\documentclass[11pt]{article}
\usepackage{amsfonts,amssymb,amsmath,amsthm}
\usepackage{enumitem}
\usepackage{pifont} 
\usepackage{xfrac} 
\usepackage{tikz}
\usepackage[textsize=tiny,shadow]{todonotes}
\usepackage{color} 
   \definecolor{cites}{rgb}{0.50 , 0.00 , 0.00}  
   \definecolor{urls} {rgb}{0.00 , 0.00 , 0.50}  
   \definecolor{links}{rgb}{0.00 , 0.00 , 0.50}   
\usepackage[
      colorlinks=true,   
      citecolor=cites,   
      urlcolor=urls,     
      linkcolor=links,   
      pdftitle={Finite sections: stability, spectral pollution and asymptotics of condition numbers and pseudospectra},  
      pdfauthor={Marko Lindner, Dennis Schmeckpeper}, 
      pdfpagemode=UseOutlines, 
      pdfstartview=FitH,       
      bookmarksopen=false      
   ]{hyperref}

\newcommand\eps\varepsilon
\newcommand\ph\varphi
\newcommand\Spec{{\rm Spec}\,}  
\newcommand\Specn{{\rm Spec}}   
\newcommand\speps{{\rm spec}_\eps}
\newcommand\Speps{{\rm Spec}_\eps}

\newcommand\Lim{{\rm Lim}}
\newcommand\Stab{{\rm Stab}}
\newcommand\dist{{\rm dist}}

\newcommand\diag{{\rm diag}}

\newcommand\closn{{\rm clos}}
\newcommand\clos{\closn\,}

\newcommand\im{{\rm im\,}}

\newcommand\Op{{\rm Op}}
\newcommand\BO{{\rm BO}}
\newcommand\BDO{{\rm BDO}}
\newcommand\BS{{\rm BS}}
\newcommand\BDS{{\rm BDS}}
\newcommand\alg{{\rm alg}}
\newcommand\prop{{\rm prop}}
\newcommand\Lay{{\rm Lay}}
\newcommand\s{s}

\newcommand\dH{d_{\rm H}}
\newcommand\ot{\leftarrow}
\newcommand\Hto{
   \unitlength0.1ex
   \begin{picture}(30,15)
   \put(13,16){\makebox(0,0)[]{\tiny\rm H}}
   \put(15,5){\makebox(0,0)[]{$\to$}}
   \end{picture}
}

\newcommand\C{{\mathbb C}}
\newcommand\R{{\mathbb R}}

\newcommand\Z{{\mathbb Z}}
\newcommand\N{{\mathbb N}}

\newcommand\cI{{\mathcal I}}
\newcommand\cK{{\mathcal K}}
\newcommand\cL{{\mathcal L}}
\newcommand\cF{{\mathcal F}}
\newcommand\cS{{\mathcal S}}
\newcommand\cN{{\mathcal N}}
\newcommand\cB{{\mathcal B}}
\newcommand\cP{{\mathcal P}}
\newcommand\cFn{{\cF_\blacktriangledown}}
\newcommand\cNn{{\cN_\blacktriangledown}}
\newcommand\ci[1]{\text{\ding{\number\numexpr#1 + 171\relax}}}

\newtheorem{theorem}{Theorem}[section]
\newtheorem{lemma}[theorem]{Lemma}
\newtheorem{corollary}[theorem]{Corollary}
\newtheorem{proposition}[theorem]{Proposition}
\newtheorem{definition}[theorem]{Definition}
\newtheorem{convention}[theorem]{Convention}

\newenvironment{example}
 {\par\noindent\refstepcounter{theorem}{\bf Example \thetheorem}\ }
 {\raisebox{1mm}{\framebox{}}\pagebreak[2]}

\newenvironment{Proof}{%
  \begin{proof}[{\bf Proof}]%
}{
   \end{proof}
}

\newtheorem{remarkx}[theorem]{Remark}
\newenvironment{remark}
  {\pushQED{\qed}\remarkx\normalfont}
  {\popQED\endremarkx}

\numberwithin{figure}{section}  

\newcounter{abccounter}
\newenvironment{abc}
  {\begin{list}{\bf \alph{abccounter})}{\usecounter{abccounter}\itemsep-1mm\topsep0mm\parskip1mm}}
  {\end{list}}

\makeatletter
\let\@fnsymbol\@arabic
\makeatother

\begin{document}
\title{\bf Finite sections: stability, spectral pollution and\\asymptotics of condition numbers and pseudospectra}
\author{
{\sc Marko Lindner}\footnote{Maths Institute, TU Hamburg, Germany, \href{mailto:lindner@tuhh.de}{\tt lindner@tuhh.de}}
\quad and\quad
{\sc Dennis Schmeckpeper}\footnote{Maths Institute, TU Hamburg, Germany,
\href{mailto:dennis.schmeckpeper@tuhh.de}{\tt dennis.schmeckpeper@tuhh.de}}
}

\maketitle
\begin{center}
{\it Dedicated to Albrecht Böttcher on his 70th birthday. Herzlichen Glückwunsch, Albrecht!}
\end{center}

\begin{quote}
\renewcommand{\baselinestretch}{1.0}
\footnotesize {\sc Abstract.}
The stability of an approximating sequence $(A_n)$ for an operator $A$ usually requires, besides invertibility of $A$, the invertibility of further operators, say $B, C, \dots$, that are well-associated to the sequence $(A_n)$.
We study this set, $\{A,B,C,\dots\}$, of so-called {\sl stability indicators} of $(A_n)$ and connect it to the asymptotics of $\|A_n\|$, $\|A_n^{-1}\|$ and $\kappa(A_n)=\|A_n\|\|A_n^{-1}\|$ as well as to spectral pollution by showing that $\limsup\Speps A_n= \Speps A\cup\Speps B\cup\Speps C\cup\dots$. We further specify, for each of $\|A_n\|$, $\|A_n^{-1}\|$, $\kappa(A_n)$ and $\Speps A_n$, under which conditions even convergence applies.
\end{quote}

\noindent
{\it Mathematics subject classification (2020):} 47A10; Secondary 47A25, 47-08.\\
{\it Keywords and phrases:} stability, condition number, spectral pollution, pseudospectrum
\section{Introduction}
For the study of a linear operator $A$ on a Banach space $X$, one often has to resort to numerical approximation: Take a sequence $(A_n)_{n\in\N}$ of simpler operators approximating $A$ in some sense, $A_n\to A$, and study the quantity of interest -- say, the inverse -- of $A_n$ as $n\to\infty$ in place of $A$.

If the approximants $A_n$ are of finite rank, the convergence $A_n\to A$ is, as a rule, never in operator norm as the uniform limit of $A_n$ were compact, ruling out, for example, all invertible operators $A$. So $A_n\to A$ is typically weaker than uniform, making the question whether also
\begin{equation} \label{eq:conv.inv}
A_n^{-1}\to A^{-1}
\end{equation}
holds largely non-trivial. If we let $\to$ refer to pointwise operator convergence on $X$, that is $\|A_nx-Ax\|\to 0$ for all $x\in X$, then the convergence \eqref{eq:conv.inv} allows for asymptotically solving $Ax=b$ via $A_nx_n=b_n$ as $n\to\infty$, see e.g.~\cite{ProeSi,BoSi1,Bo94,BoSi2,HaRoSi2,RaRoSiBook,BoGru,Li:Book,SeiDiss,HagLiSei}.
\medskip

{\bf Stability.\ }
The key to \eqref{eq:conv.inv} is the so-called {\sl stability} of the sequence $(A_n)$, typically yielding
\\[-7mm]
\begin{itemize} \itemsep-1mm
\item existence of $A^{-1}$,
\item convergence of $A_n^{-1}$,
\item ... to $A^{-1}$.
\end{itemize}

By definition, an operator sequence $(A_n)_{n\in\N}$ is {\sl stable} if all but finitely many $A_n$ are invertible and their inverses are uniformly bounded. We can express this as follows:
\begin{equation} \label{eq:def.stab}
(A_n) \text{ is stable }
\quad:\!\iff\quad \limsup
\|A_n^{-1}\| < \infty,
\end{equation}
where we put $\|B^{-1}\|:=\infty$ if and only if $B$ is not invertible.

In $\S6$ of \cite{HagLiSei}, looking at the same equivalence \eqref{eq:def.stab}, the authors ask about quantities:

\noindent
\begin{tabular}{p{5mm}ll}
&(Q1)&How large is the $\limsup$ in \eqref{eq:def.stab}?\\ 
&(Q2)&Is it possibly a limit? \\
&(Q3)&What is the asymptotics of the condition numbers, $\kappa(A_n)=\|A_n\|\cdot\|A_n^{-1}\|$?\\
&(Q4)&What is the asymptotics of the pseudospectra of $A_n$?
\end{tabular}

We take this as a program for our paper here.
\medskip

{\bf Our results.\ }
For sequences $(A_n)$ in the finite section algebra of band-dominated (which are bounded but generally non-normal) operators $A$ on $X=\ell^p(\Z)$ with $p\in[1,\infty]$, we associate a set 
\[
\Stab\big(\,(A_n)_{n\in\N}\,\big)\ =:\ \Stab(A_n)
\]
of operators on $X$ -- the corresponding {\sl stability indicators} -- with each sequence $(A_n)$, with the property that $(A_n)$ is stable if and only if every $B\in\Stab(A_n)$ is invertible. We demonstrate how that same set, $\Stab(A_n)$, determines precise answers to each of (Q1) -- (Q4):
\begin{equation} \label{eq:results}
\left.
\begin{array}{rcl}
\limsup\|A_n\|&=&\max\limits_{B\in\Stab(A_n)}\|B\|\\[4mm]
\limsup\|A_n^{-1}\|&=&\max\limits_{B\in\Stab(A_n)}\|B^{-1}\|\\[4mm]
\limsup\Speps A_n&=&\bigcup\limits_{B\in\Stab(A_n)}\Speps B,\qquad\eps>0.
\end{array}
\qquad\right\}
\end{equation}
Via certain triples $(A,B,C)$ in $\Stab(A_n)$, we get a formula for $\limsup\kappa(A_n)$ and show, for each of our quantities, how also the answer to whether or not $\limsup=\lim$ is encoded in $\Stab(A_n)$.
\medskip

{\bf A remark on spectral pollution.\ }
In an ideal world, one would hope that $A_n\to A$ generally implies $\lim \Speps A_n=\Speps A$.
But the truth is: neither does the set sequence $\Speps A_n$ converge in Hausdorff sense as $n\to\infty$, whence we say ``$\limsup$'' instead of ``$\lim$'' in \eqref{eq:results}, nor is the result equal to $\Speps A$. Instead it is the union of $\Speps A$ with $\Speps B$ for all other stability indicators $B$ of $(A_n)$ (note that $A$ is always an element of $\Stab(A_n)$). So if, in practical computations, $\Speps A_n$ is found to approximate points that are far away from $\Speps A$, so-called {\sl spectral pollution}, then the last formula in \eqref{eq:results} exactly says who is to blame for this: some of the other $B\in\Stab(A_n)$ and their pseudospectra. A priori knowledge of $\Stab(A_n)$ may hence also help to classify and ignore spectral pollution.
\medskip

{\bf The operators.\ }
Thinking of an operator $A$ on $X:=\ell^p(\Z)$ with $p\in[1,\infty]$ as a bi-infinite matrix $(A_{ij})_{i,j\in\Z}$, we call $A$ a {\sl band operator}, $A\in\BO$, if its matrix is supported on finitely many diagonals only, and if $A$ is bounded as an operator on $X$. Then let $\BDO$, the set of {\sl band-dominated operators}, denote the closure of $\BO$ in the operator norm topology.
\medskip

{\bf Approximants: Pure and composed finite sections.\ }
For reasons of finite storage and computational cost, we approximate $A$, as $n\to\infty$, by finite matrices $A_n=(F_{ij})_{i,j=-n}^n$, each interpreted, via zero extension, as an operator on $\ell^p(\Z)$. For the precise construction of $A_n$ with $n\in\N$, let $P_n$ denote the operator of multiplication by the characteristic function of $\{-n,\ldots,n\}$ and suppose, as an example, that $A=BC+D$ with somehow simpler operators $B,C,D\in\BDO$. Then we either take\\[-7mm]
\begin{itemize}\itemsep0mm
\item {\sl pure finite sections}: $\ A_n:=P_nAP_n$ -- cutting $(A_{ij})_{i,j=-n}^n$ out of $A$, extended by zeros,\ or
\item {\sl composed finite sections}: $\ A_n:=(P_nBP_n)(P_nCP_n)+(P_nDP_n)$ -- the sum-product of pure finite sections corresponding to the decomposition $A=BC+D$.
\end{itemize}
The latter is an illustration of something much bigger: the set $\cS$ of all limits $(A_n)$ of finite sum-products of pure finite section sequences $(P_nBP_n)$ with $B\in\BDO$.
To make sense of the words ``limit'', ``sum'' and ``product'' in this operator sequence context, we introduce, in Section~\ref{sec:tools} below, a Banach algebra $\cF$ of bounded operator sequences together with the two-sided ideal, $\cN$, of null sequences in $\cF$. We will see that, in a sense, the quotient algebra $\cF/\cN$ is the playground for algebraic studies of stability.

We wish to add that cutting finite sections out of an infinite matrix is much more than a toy model. See e.g.~\cite{LiSeif:FSM} for why celebrated methods like Galerkin, Ritz-Galerkin, FEM and PCE are not more than special cases of pure finite sections in $\ell^2(\N)$.
\medskip

{\bf Semi-infinite matrices and their finite sections.\ }
Our focus on bi-infinite matrices $A$ and finite sections $A_n$ taken from $-n$ to $n$ is not a big restriction. If one is instead interested in a semi-infinite matrix $B$ and its pure finite sections $B_n$ from $1$ to $n$ then let $A$ be the bi-infinite extension of $B$ by $c$ times the identity, where $c>\|B\|$, so that $A_n$ is invertible if and only if $B_n$ is invertible and $\|A_n^{-1}\|=\max(c^{-1},\|B_n^{-1}\|)=\|B_n^{-1}\|$ since $c^{-1}<\|B_n\|^{-1}\le\|B_n^{-1}\|$ by $c>\|B\|\ge\|P_nBP_n\|=\|B_n\|$. So the stability of $(B_n)$ and the asymptotics of $\|B_n^{-1}\|$ are the same as those of $(A_n)$ and $\|A_n^{-1}\|$, respectively. The $\eps$-pseudospectra of $A_n$ are of course those of $B_n$ together with an $\eps$-ball around $c$. But the latter is isolated and easily ignored if $c$ is sufficiently large.
\medskip

{\bf Banach space-valued $\ell^p$ over $\Z^d$ with $p\in [1,\infty]$.\ }
We stick to scalar-valued $\ell^p(\Z)$ with $p\in(1,\infty)$ for much of the exposition, avoiding machinery that could easily hide the true plot of the paper. We explain how to deal with Banach space-valued $\ell^p(\Z^d)$ with $p\in [1,\infty]$ in Section~\ref{sec:Banach}.

Note that restriction to just $p=2$ in this earlier part would be a bit too narrow as it would motivate the following elegant argument (that we learnt from \cite{Bo94}) which is however not generalizable to $p\ne 2$: Since $(A_n)+\cN\mapsto \Stab(A_n)$ is an injective $~^*$-homomorphism between suitable $C^*$-algebras, preserving invertibility, it automatically preserves norms. The analysis of \cite{HagLiSei} shows norm-preservation by other means, not limited to $p=2$. 
\medskip

{\bf History: operator algebraists and stability.\ }
The study of stability, naturally claimed as home territory by numerical analysts, nevertheless carried out by means of operator algebraic arguments can be traced back to at least the 1970s: 

In 1974, Kozak showed in his PhD thesis \cite{Koz} that a bounded sequence $(A_n)$ is stable if and only if its coset is invertible in the quotient algebra $\cF/\cN$, the operator version of $\ell^\infty/c_0$.
This opened an extremely powerful and productive back door into numerical analysis for operator algebraists, see \cite{KozSim,ProeSi,BoSi1,Bo94,BoSi2,HaRoSi2,RoSantSi,BoGru} and many others. 

A related but not quite identic approach to stability is the following: In 1971 Douglas and Howe \cite{DouglasHowe} and later Gorodetski \cite{Goro} realized, in their study of multidimensional discrete convolutions with homogeneous symbol, that an approximating sequence $(A_n)$ of growing finite matrices is stable if and only if the direct sum, $\oplus A_n$, is a Fredholm operator. 
In fact, this approach is not too different from Kozaks: $\oplus A_n$ is a Fredholm operator if and only if its coset is invertible in the quotient algebra $L(X)/K(X)$ of bounded operators modulo compact operators on $X=\ell^p(\Z)$ (the so-called Calkin algebra), and for $\oplus A_n$ this perfectly translates to invertibility in $\cF/\cN$.

Later Lange and Rabinovich  \cite{LaRa} started a systematic study of Fredholm properties in $\BDO$ by so-called {\sl limit operators}. (This interplay and its consequences have been strengthened and simplified in \cite{RaRoSi1998,RaRoSiBook,Li:Book,CWLi:Memoir,Sei:Survey} and somehow finalized in \cite{LiSei:BigQuest,HagLiSei}.
Our approach here largely rests on these last two papers.) 
The main message is that $A\in\BDO$ is a Fredholm operator if and only if all its limit operators -- capturing the behaviour of $A$ at infinity in all kinds of directions -- are invertible. Denoting the family of all limit operators of $A$ by $\Lim(A)$, one can obviously connect these Fredholm studies with Douglas, Howe and Gorodetski's $\oplus A_n$ construction, showing that an operator sequence $(A_n)$ is stable if and only if every operator in $\Lim(\oplus A_n)$ is invertible. Removing some more or less obvious redundancies from $\Lim(\oplus A_n)$ if $(A_n)\in\cS$ then leads to the set $\Stab(A_n)$ that takes center stage in our paper.
\medskip

{\bf More history: stability indicators.\ }
Kozaks PhD thesis \cite{Koz} is a fore-runner also here. For convex polygons $\Omega\subset\R^2$ with vertices $v_1,\ldots,v_k\in\Z^2$, Kozak looked at the pure finite sections $A_n$ with respect to $(n\Omega)\cap\Z^2$ of discrete convolutions $A$ in $\ell^2(\Z^2)$. He showed that $(A_n)$ is stable if and only if the operators $B_1,\ldots, B_k$ are invertible, where $B_j$ is the compression of $A$ to the infinite cone that fits $\Omega$ at its corner $v_j$.
Böttcher and Silbermann \cite{BoSi1,BoSi2} famously took the 1D version of Kozaks result to both the algebra of Toeplitz operators and their finite section algebra, ending up with two stability indicators -- one for each endpoint of the interval. Moreover, they already arrived at all of the formulas \eqref{eq:results} for their particular situation.
Rabinovich, Roch and Silbermann \cite{RaRoSiBook} then even replaced the Toeplitz algebra by $\BDO$. 
Important other work is, for example, \cite{MascSantSei}, where, though limited to Hilbert space, the operator setting is even more general than $\BDO$ and \cite{SeiDiss,SeiSilb13}, where, based on deep localization techniques of \cite{SeiDiss}, all of \eqref{eq:results} are shown in the finite section algebra $\cS$ over $\BDO$.
\medskip

{\bf So what is different in our paper?\ }
Thanks to the quantitative results of \cite{HagLiSei}, we can leave the Hilbert space / C$^*$-algebra setting of \cite{RaRoSiBook,MascSantSei} behind. Thanks to \cite{LiSei:BigQuest,HagLiSei}, we can replace \cite{RaRoSiBook}'s $\sup$'s by $\max$'s, drop a closure in the pseudospectral formula and hugely simplify the proofs. Based on \cite{RaRoSi:FSMsubs} and \cite{Li:FSMsubs}, we can transfer the formulas \eqref{eq:results} to subsequences $(A_{n_k})$ and thereby elegantly identify subsequences with a different $\limsup$ or, in case of their absence, conclude that the $\limsup$ is a proper limit. And finally, our approach via $\Lim(\oplus A_n)$, arguably simplifying the proofs and techniques, allows to generalize the results of \cite{SeiDiss,SeiSilb13} from $(A_n)\in\cS$ to the much larger sequence algebra $\BDS$.

\section{Two examples}
Before we give the proper details, let us get into the right spirit via two examples.

\begin{example} \label{ex:Laurent} {\bf (Laurent operator)\ }
We start with a banded Laurent operator, that is, $L=(L_{ij})_{i,j\in\Z}=(a_{i-j})_{i,j\in\Z}$ with constant diagonals, finitely many of them nonzero, i.e.~$a\in c_{00}(\Z)$, acting boundedly on $\ell^p(\Z)$. For the pure finite sections, $L_n=P_nLP_n$, it is easily shown that
\[
\Stab(L_n)=\{L,T,\tilde T\} \quad\text{with}\quad 
T=(L_{ij})_{i,j\in\N_0},\quad\tilde T=(L_{ij})_{i,j\in-\N_0}\quad\text{and}\quad \N_0=\N\cup\{0\}.
\]
\begin{minipage}{117mm}
$T$ is what $L_n$ (interpreted as the  finite matrix $(L_{ij})_{i,j=-n}^n$) asymptotically, as $n\to\infty$, looks like from the perspective of its top left corner, and $\tilde T$ is the same kind of limit when watching $L_n$ grow from its lower right corner. $L$ is the limit when focussing at $(0,0)$, while $n\to\infty$.
We can deal with this mix of bi-infinite and semi-infinite matrices in one set, also with $T$ and $\tilde T$ facing in different directions.
\end{minipage}
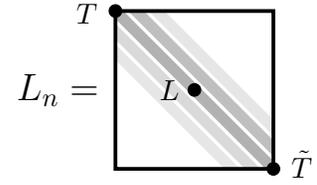
\begin{minipage}{43mm}
\begin{flushright}
\begin{tikzpicture}[scale=0.21]
\fill[gray!60] (0,9.5) -- (0,10) -- (0.5,10) -- (10,0.5) -- (10,0) -- (9.5,0);
\fill[gray!30] (0,8.2) -- (0,9.2) -- (9.2,0) -- (8.2,0);
\fill[gray!20] (0,6.9) -- (0,7.9) -- (7.9,0) -- ( 6.9,0);
\fill[gray!50] (0.8,10) -- (1.8,10) -- (10,1.8) -- (10,0.8);
\fill[gray!20] (2.1,10) -- (3.1,10) -- (10,3.1) -- (10,2.1);
\node at (-0.4,5)[left] {\Large $L_n=$};
\draw [black,line width=0.5mm] (0,0) rectangle (10,10); 
\draw[fill] (0,10) circle (4mm); \node at (-0.5,9.8)[left] {$T$};
\draw[fill] (10,0) circle (4mm); \node at (10.5,0.3)[right] {$\tilde T$};
\draw[fill] (5,5) circle (4mm); \node at (4.7,5)[left] {$L$};
\end{tikzpicture}
\end{flushright}
\end{minipage}

So $(L_n)$ has three stability indicators: $L$, $T$ and $\tilde T$. Their invertibility is sufficient and necessary for the stability of $(L_n)$. 
$L$ itself is in fact redundant in this set: $\|L\|\le\|T\|$ (actually equality) and $\|L^{-1}\|\le\|T^{-1}\|$, so that invertibility of $T$ implies that of $L$. We will say that $L$ is dominated\footnote{Between $T$ and $\tilde T$, the situation is not so clear: Although, as is easily seen, $\|T\|=\|L\|=\|\tilde T\|$ holds, it is possible that $\|T^{-1}\|\ne\|\tilde T^{-1}\|$,
see \cite[Ex.~6.6]{BoGru}. This is because, after flipping both rows and columns around, $\tilde T$ is the transpose of $T$, acting on $\ell^q$ with $1/p+1/q=1$, not on $\ell^p$ (unless $p=q=2$). } by $T$. Also $L-\lambda I$ is dominated by $T-\lambda I$ for $\lambda\in\C$, whence $\Speps L\subseteq\Speps T$. All of these are limit operator arguments, see below.

Our formulas  \eqref{eq:results}, reproducing well-known results of \cite{Bo94,BoSi2,HaRoSi2,BoGru}, show that $\|L_n\|\to\|L\|$ and
\begin{eqnarray*}
\limsup\|L_n^{-1}\| &=& \max( \|L^{-1}\|, \|T^{-1}\|,\ \|\tilde T^{-1}\|)\ =\ \max( \|T^{-1}\|,\ \|\tilde T^{-1}\|),\\
\limsup\kappa (L_n) &=& \|L\|\cdot \max( \|T^{-1}\|,\ \|\tilde T^{-1}\|)\ =\ \max(\kappa(T),\ \kappa(\tilde T)),\\
\limsup\Speps L_n &=& \Speps L\ \cup\ \Speps T\ \cup\ \Speps \tilde T\ =\ \Speps T\ \cup\ \Speps \tilde T.
\end{eqnarray*}
By Example \ref{ex:L-FSalg}, reproducing results of \cite{Bo94,BoSi2,HaRoSi2,BoGru} for sequences in the algebra of finite sections of banded Laurent operators, all three $\limsup$'s are proper limits. 
In case $p=2$, $\|T^{-1}\|=\|\tilde T^{-1}\|$ and $\Speps T=\Speps\tilde T$ hold, leading to further simplifications.
\end{example}

\begin{example} \label{ex:symm_blockflip} {\bf (symmetric block-flip)\ }
Our next example, also see \cite[Ex.\,6.3]{HagLiSei}, is 
\[
F\ =\ \diag(\cdots,B,B,\boxed 1,B,B,\cdots),\quad\textrm{where\ }
B\ =\ \left(\begin{array}{cc}\mu&1\\1&\mu\end{array}\right),
\quad\mu\in[0,1)
\quad\text{and}\quad\boxed 1=F_{00}.
\]
The pure finite sections, $F_n=P_nFP_n$, correspond to the finite $(2n+1) \times (2n+1)$ matrices
\[
F_{n}\ =\ \left\{
\begin{array}{cl}
\diag\big(B,\cdots,B,\boxed 1,B,\cdots,B\big) & \textrm{if $n\ge 2$ is even},\\
\diag\big(\mu,B,\cdots,B,\boxed 1,B,\cdots,B, \mu\big) & \textrm{if $n\ge 3$ is odd.}
\end{array}\right.
\]
Looking at the singular values, $1-\mu$ and $1+\mu$, of $B$, we conclude that $\|B\|=\|B\|_2=1+\mu$ and $\|B^{-1}\|=\|B^{-1}\|_2=(1-\mu)^{-1}$, whence,
\[
\begin{array}{|c|c|}
\hline
\text{if $n\ge 2$ is even then} &\text{if $n\ge 3$ is odd then}\\
\hline
\rule{55mm}{0pt} & \rule{95mm}{0pt}\\[-3mm]
\|F_{n}^{-1}\|=\|B^{-1}\|=(1-\mu)^{-1} ,& \|F_{n}^{-1}\|=\max\left(\|B^{-1}\|,\mu^{-1},1\right)=\max\left((1-\mu)^{-1},\mu^{-1}\right),\\[1mm]
\kappa(F_n)=(1+\mu)(1-\mu)^{-1}, & \kappa(F_n)=(1+\mu)\max\left((1-\mu)^{-1},\mu^{-1}\right),\\[1mm]
\Spec F_n=\{\mu\pm 1,\,1\}, & \Spec F_n=\{\mu\pm 1,\,1,\,\mu\}.\\[1mm]
\hline
\end{array}
\]
The sequences $\|F_n^{-1}\|$ and $\kappa(F_n)$ are convergent (in fact, constant) if and only if $\mu^{-1}\le(1-\mu)^{-1}$, i.e.~if $\mu\in [\frac12,1)$. The pseudospectra are the $\eps$-neigbourhoods of the spectra (selfadjoint case).

In this example, $(F_n)$ has five stability indicators; they are 
$\Stab(F_n)=\{F,C,D,\tilde C,\tilde D\}$ with
\[
C=\diag(B,B,\dots),\ D=\diag(\mu,B,B,\dots),\ \tilde C=\diag(\dots,B,B),\ \tilde D=\diag(\dots,B,B,\mu),
\]
so that $\|F^{-1}\|=\|B^{-1}\|=\|C^{-1}\|=\|\tilde C^{-1}\|=(1-\mu)^{-1}$ and $\|D^{-1}\|=\|\tilde D^{-1}\|=(\min(1-\mu,\mu))^{-1}$.
The limits $C$ and $\tilde C$ arise from the perspective of the top left, resp.~lower right, corner of the matrices $(F_{2k})$ as $k\to\infty$. $D$ and $\tilde D$ are the same limits for $(F_{2k+1})$ as $k\to\infty$. In Section~\ref{sec:subseq} below we associate the subset $\{F,C,\tilde C\}$ of $\Stab(F_n)$ with the subsequence $(F_{2k})$ of $(F_n)$ and $\{F,D,\tilde D\}$ with the subsequence $(F_{2k+1})$.

Note that these subsets of stability indicators and their norms and norms of inverses help to quantify the $\limsup$'s of $\|F_n\|$, $\|F_n^{-1}\|$ and $\kappa(F_n)$ 
for the corresponding subsequences and hence to tell stable from unstable subsequences (look at even vs.~odd $n$ in the case $\mu=0$).
\end{example}

\section{Tools and notations} \label{sec:tools}

{\bf Lower norm, spectrum and pseudospectrum.\ }
Let $A$ be a bounded linear operator on a Banach space $X$.
A fairly convenient access to the norm of $A^{-1}$ is by the so-called {\sl lower norm}, the number
\begin{equation}\label{eq:nu}
\nu(A)\ :=\ \inf_{\|x\|=1}\|Ax\|.
\end{equation}
Indeed, putting 
\begin{equation}\label{eq:invmu}
\mu(A):=\min\{\nu(A),\,\nu(A^*)\},
\qquad\text{we have}\qquad
\|A^{-1}\|\ \ =\ 1/\mu(A),
\end{equation}
where $A^*$ is the Banach space adjoint on the dual space $X^*$ and equation \eqref{eq:invmu} takes the form $\infty=1/0$ if and only if $A$ is not invertible. 

This enables us to write the {\sl pseudospectra} \cite{TrefEmb} of $A$ as sublevel sets of $\mu$,
\[
\textstyle
\speps A\ :=\ \{\lambda\in\C: \|(A-\lambda I)^{-1}\|>\frac 1\eps \}
\ =\ \{\lambda\in\C: \mu(A-\lambda I)<\eps \},\qquad \eps>0
\]
and
\[
\textstyle
\Speps A\ :=\ \{\lambda\in\C: \|(A-\lambda I)^{-1}\|\ge\frac 1\eps \}
\ =\ \{\lambda\in\C: \mu(A-\lambda I)\le\eps \},\qquad \eps\ge 0.
\]
In particular, the {\sl spectrum} of $A$ is part of this family:
\[
\Spec A\ =\ \Specn_0 A\ \subset \speps A\ \subset\ \Speps A\ =\ \clos\,\speps A,\qquad \eps>0,
\]
where the last equality holds by results of Globevnik \cite{Globevnik} and Shargorodsky \cite{Shargorodsky08} if $X=\ell^p(\Z)$.
\medskip

{\bf Set sequences, Hausdorff-convergence and spectral pollution.\ }

For two bounded nonempty sets, $S,T\subset\C$, the expression
\[
\dH(S,T)\ :=\ \max\left( \sup_{s\in S} \dist(s,T)\,,\, \sup_{t\in T}\dist(t,S)\right),
\quad\text{where}\quad
\dist(s,T):=\inf_{t\in T}|s-t|,
\]
denotes the {\sl Hausdorff distance} of $S$ and $T$. On the set of all compact subsets of $\C$, $\dH$ is a metric; on the bounded subsets of $\C$ it is merely a pseudometric since $\dH(S,T)=\dH(\clos S,T)$, so that $\dH(S,T)=0$ iff $\clos S=\clos T$ but not necessarily $S=T$. 

One says that $T_n$ {\sl Hausdorff-converges} to $T$, written $T_n\Hto T$, if $\dH(T_n,T)\to 0$, noting that the limit $T$ is only unique after passing to its closure. 

For a bounded sequence $(T_n)$, Hausdorff-convergent or not, look at the sets $\limsup T_n$, the set of all partial limits of sequences $(t_n)$ with $t_n\in T_n$, and $\liminf T_n$, the set of all limits of sequences $(t_n)$ with $t_n\in T_n$. By the Hausdorff theorem, e.g.\cite[\S3.1.2]{HaRoSi2}, $T_n\Hto T$ if and only if
\[
\liminf T_n\ =\ \limsup T_n\ =\ \clos T.
\]
When aiming to approximate a set $T$ by sets $T_n$, the Hausdorff distance $\dH(T_n,T)$
detects two kinds of failure: $(i)$ $T_n$ failing to approximate parts of $T$ and $(ii)$ $T_n$ clustering at points far from $T$. The latter is known as {\sl spectral pollution}.

{\bf Some helpful short notations.\ }
For $a,b\in\Z$ with $a\le b$, let us write
\[
a..b\ :=\ \{z\in\Z:a\le z\le b\},\quad
a..\ :=\ \{z\in\Z:a\le z\},\quad\text{and}\quad
..b\ :=\ \{z\in\Z: z\le b\}.
\]
For a set $M\subset\Z^d$, let $P_M$ denote the operator on $\ell^p(\Z^d)$ that mulitplies by the characteristic function of $M$.
In particular, for $d=1$, put $P_+:=P_{0..}$, $P_-:=P_{..0}$ and $P_n:=P_{-n..n}$ with $n\in\N$.
\medskip

{\bf Algebras of operators and matrices.\ }
Identifying operators on $X:=\ell^p(\Z^d)$, $d\in\N$, with $\Z^d\times\Z^d$-matrices is almost straightforward: 
With an operator $A$ on $X$, associate the infinite matrix $[A]=(A_{ij})_{i,j\in\Z^d}$
with $A_{ij}\in\C$ given by $P_{\{i\}}AP_{\{j\}}:\im P_{\{j\}}\to \im P_{\{i\}}$. 
Conversely,  a bi-infinite matrix $M=(M_{ij})_{i,j\in\Z^d}$ induces, via matrix-vector multiplication, an operator $\Op(M):(x_j)_{j\in\Z^d}\mapsto (\sum_{j\in\Z^d} M_{ij} x_j)_{i\in\Z^d}$, provided the sums converge. For certain classes of matrices $M$, $\Op(M)$ is a bounded operator on $X$, and $[\Op(M)]=M$. For $p<\infty$, also $\Op([A])=A$, see \cite[Prop.\,1.31 but also Ex.\,1.26 c]{Li:Book}. Now let us put
\begin{eqnarray*}
\BO &:=& \alg\{S_k,\, M_b\ :\ k\in\Z^d,\, b\in\ell^\infty(\Z^d)\}\qquad\text{and}\\
\BDO &:=& \clos_{L(X)} BO(X)\ =\ \clos\alg\{S_k,\, M_b\ :\ k\in\Z^d,\, b\in\ell^\infty(\Z^d)\},
\end{eqnarray*}
where $S_k$ is the $k$-shift on $X$ with $(S_kx)_{i+k}=x_i$ for $i\in\Z^d$,
$M_b$ is the operator of multiplication by $b$, $(M_bx)_i=b_ix_i$,
and $\alg\{\dots\}$ refers to the set (in fact, the algebra) of all finite sum-products of these.
We call elements of $\BO$ and $\BDO$, respectively, {\sl band operators} and {\sl band-dominated operators} \cite{RaRoSiBook,Li:Book}.
Moreover, let
\[
\prop(A)\ :=\ \sup\{|i-j|_\infty\ :\ i,j\in\Z^d,\ P_{\{i\}}AP_{\{j\}}\ne 0\}
\]
denote the {\sl propagation} of $A\in\BO$, a.k.a.~the {\sl bandwidth} of the corresponding matrix $[A]$.
\medskip

{\bf Algebras of operator sequences.\ }
Let $\cF$ be the set of all bounded sequences $(A_n)_{n\in\N}$ of bounded operators $A_n$ on $X:=\ell^p(\Z)$, and equip $\cF$ with elementwise operations, $\alpha(A_n):=(\alpha A_n)$, 
$(A_n)+(B_n):=(A_n+B_n)$ and $(A_n)(B_n):=(A_nB_n)$, and with norm $\|(A_n)\|_\cF:=\sup_n\|A_n\|$, turning $\cF$ into a Banach algebra. Let $\cN$ denote the closed ideal in $\cF$ of all sequences $(A_n)\in\cF$ with $\|A_n\|\to 0$. 
Ultimately, we want the $A_n$ to be finite matrices growing with $n$, so let us put
\begin{eqnarray*}
\cFn&:=&\{(P_nA_nP_n)_{n\in\N}\ :\ (A_n)\in\cF\},\\
\cNn&:=&\{(A_n)\in\cFn\ :\ \|A_n\|\to 0\},
\end{eqnarray*}
where now $A_n$ is interpreted as operator on $\ell^p(-n..n)$ when it comes to invertibility and inverses.
In analogy to the operator algebras $\BO$ and $\BDO$, now let
\begin{eqnarray*}
\BS&:=& \alg\{(S_k)_{n\in\N},\ (M_{b^{(n)}})_{n\in\N}\ :\ k\in\Z,\ b^{(n)}\in\ell^\infty(\Z),\ \sup\|b^{(n)}\|_\infty<\infty\}\ \cap\ \cFn,\\
\BDS&:=& \closn_\cF\,\BS\ =\ \clos\alg\{(S_k),\ (M_{b^{(n)}})\ :\ \dots\}\ \cap\ \cFn
\end{eqnarray*}
refer to the algebras of all {\sl band sequences} and all {\sl band-dominated sequences}, respectively. Being finite sum-products of $(S_k)_n$ and $(M_{b^{(n)}})_n$, band sequences $(A_n)\in\BS$ are exactly those bounded sequences of band operators $A_n\in\BO$ with the property that $\sup_n\prop(A_n)<\infty$. Finally, let
\[
\cS\ :=\ \clos\alg\{(P_nBP_n)\ :\ B\in\BDO\}
\]
denote the so-called {\sl finite section algebra} announced in the introduction.
$(A_n)\in\cS$ means that
\begin{equation} \label{eq:(An)alg}
(A_n)_{n\in\N}\ =\ \lim_{i\to\infty}\sum_{j\in J_i}\prod_{k\in K_i} (P_nA^{(i,j,k)}P_n)_{n\in\N}
\end{equation}
with finite sets $J_i,K_i\subseteq\N$ for each $i\in\N$, with topology, sum and product of $\cF$ and with all $A^{(i,j,k)}\in\BDO$. For each individual $A_n$ that means
\begin{equation} \label{eq:Analg}
A_n\ =\ \lim_{i\to\infty}\sum_{j\in J_i}\prod_{k\in K_i} (P_nA^{(i,j,k)}P_n),\qquad n\in\N,
\end{equation}
now with topology, sum and product of $L(X)$ and with the convergence $\lim_{i\to\infty}$ uniform in $n$.
Then $A_n\to A:=\lim_i\sum_j\prod_k A^{(i,j,k)}$, independently of the representation of $(A_n)$.
\begin{lemma}
It holds that $\cS\subset\BDS$.
\end{lemma}
\begin{Proof}
By the definitions of $\cS$, $\cFn$ and $\BDO$, $\cS\subset\cFn$ holds as well as
\begin{eqnarray*}
\cS&\subset&\clos\alg\{(P_n)_n,(S_k)_n,(M_b)_n\ :\ k\in\Z,\ b\in\ell^\infty(\Z)\}\\
&\subset&\clos\alg\{(S_k)_n,(M_{b^{(n)}})_n\ :\ k\in\Z,\ b^{(n)}\in\ell^\infty(\Z),\,\sup\|b^{(n)}\|_\infty<\infty\}.
\qedhere
\end{eqnarray*}
\end{Proof}
\medskip

{\bf Fredholm operators, Calkin algebra and limit operators.\ }
Recall that a bounded linear operator $A$ on a Banach space $X$ is a {\sl Fredholm operator}
if its coset, $A+K(X)$, modulo compact operators $K(X)$, is invertible in the so-called {\sl Calkin algebra} $L(X)/K(X)$. This holds if and only if the nullspace of $A$ has finite dimension and the range of $A$ has finite codimension in $X$. In particular, Fredholm operators have a closed range. 

Since the coset $A+K(X)$ cannot be affected by changing finitely many matrix entries, its study takes place ``at infinity''. This is where limit operators \cite{RaRoSi1998,RaRoSiBook,Li:Book,LiSei:BigQuest} come in:

\begin{definition}
For $A\in\BDO$ on $X=\ell^p(\Z^d)$ with $p\in(1,\infty)$ and $d\in\N$, we look at all its translates $S_{-k}AS_k$ with $k\in\Z^d$ and speak of a {\sl limit operator}, $A_h$, on $\ell^p(\Z^d)$ if, for a particular sequence $h=(h_n)$ in $\Z^d$ with $|h_n|\to\infty$, the corresponding sequence of translates converges pointwise to $A_h$, that is, $S_{-h_n}AS_{h_n}\to A_h$ as $n\to\infty$.

Moreover, let $\Lim(A)$ denote the set of all limit operators of $A$, together with the following local versions: For a sequence $h=(h_n)$ in $\Z^d$ with $|h_n|\to \infty$, we put 
$$
\Lim_h(A)\ :=\ \{A_g\ :\ g\text{ is a subsequence of } h\}.
$$
For $d=1$ we fix the special cases $\Lim_+(A):=\Lim_{(1,2,\ldots)}(A)$ and $\Lim_-(A):=\Lim_{(-1,-2,\ldots)}(A)$.
\end{definition}

By repeated application of the Bolzano-Weierstra\ss\ theorem, it is shown \cite{RaRoSi1998,RaRoSiBook,Li:Book} that $\Lim_h(A)\ne\varnothing$ for all $A\in\BDO$ and all sequences $h=(h_n)$ in $\Z^d$ with $|h_n|\to \infty$ if $X$ is a scalar-valued $\ell^p$ space. For Banach space-valued $\ell^p$ spaces, this introduces an additional condition on $A$, see Section \ref{sec:Banach} below.

For $A\in\BDO$, the identification between a coset $A+K(X)$ in the Calkin algebra and the set $\Lim(A)$ preserves algebra operations, invertibility and inverses, hence spectra, but also norms, hence pseudospectra. 
See \cite{LaRa,RaRoSi1998,RaRoSiBook,Li:Book,CWLi:Memoir,Sei:Survey,LiSei:BigQuest,HagLiSei} for the key steps.
\medskip

{\bf The stacked operator $\oplus A_n$.\ } 
Let $(A_n)\in\cFn$. There are different ways to assemble a sequence of growing finite square matrices $A_n$ via a direct sum. One way is the classical block diagonal matrix, acting on $\ell^p(\N)$ or $\ell^p(\Z)$. Following \cite{RaRoSiBook,Li:Book, HagLiSei}, we take an alternative approach, where each $A_n$ remains an operator on $\ell^p(\Z)$: Take $u=(u_{m,n})_{m,n\in\Z}\in\ell^p(\Z^2)$ and let $A_n$ act on $(u_{\,\cdot,n})\in\ell^p(\Z)$ for each $n\in\N$. Because that way, each $A_n$ acts on its own invariant subspace of $\ell^p(\Z^2)$, the sequence $(A_n)$ acts as a direct sum:\\
\noindent
\begin{minipage}{90mm}
\[
(\oplus A_n\ u)_{m,n}\ :=\ (A_n (u_{\,\cdot,n}))_m\,,\quad m\in\Z,\ n\in\N.
\]
We refer to $A_n$ as the $n$-th {\sl layer} of the operator $\oplus A_n$ defined above.
For completeness, put $A_n:=0$ for $n\in\Z\setminus\N$ and extend the construction of $\oplus A_n$ to all $n\in\Z$. It follows that 
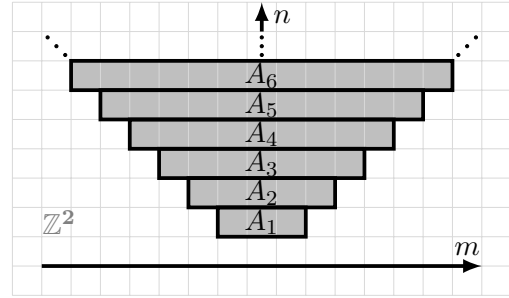
\begin{equation} \label{eq:normoplus}
\|\oplus A_n\|\ =\ \sup_n\|A_n\|\ =\ \|(A_n)\|_\cF
\end{equation}
and that $\oplus A_n$ is invertible if and only if every $A_n$ is invertible and their inverses are uniformly bounded. 
\end{minipage}
\begin{minipage}{70mm}
By $(A_n)\in\cFn$, we get the following support pattern for $\oplus A_n$:
\begin{center}
\begin{tikzpicture}[scale=0.39]
\foreach \k in {1,2,...,6}
  \fill[gray!50] (-\k,\k) rectangle (\k+1,\k+1);
\draw[gray!30,very thin,step=1cm] (-8,-1) grid (9,9);
\foreach \k in {1,2,...,6} {
   \node at (0.5,{\k+0.5}) {$A_{{\k}}$};
   \draw [black,line width=0.5mm] (-\k,\k) rectangle (\k+1,\k+1); 
}   
\foreach \j in {1,2,3} {
  \filldraw (0.5,{6.9+0.3*\j}) circle (1.5pt);
  \filldraw ({6.9+0.3*\j},{6.9+0.3*\j}) circle (1.5pt);  
  \filldraw ({-5.9-0.3*\j},{6.9+0.3*\j}) circle (1.5pt);  
}
\draw[-latex,line width=0.5mm] (0.5,8.0) -- (0.5,9);
\node at (1.2,8.5) {$n$};
\draw[-latex,line width=0.5mm] (-7,0) -- (8,0);
\node at (7.5,0.6) {$m$};
\node [text=gray] at (-6.4,1.5) { $\bf \Z^2$};
\end{tikzpicture}
\end{center}
\end{minipage}

Let $\Lay(\oplus A_n)$ denote the set of all layers of $\oplus A_n$, i.e.~$\{A_n:n\in\N\}$, and carry the notation over to sets such as $\Lay(\Lim(\oplus A_n)):=\cup\ \Lay(\oplus B_n)$, the union taken over all $\oplus B_n\in\Lim(\oplus A_n)$.
\begin{lemma} \label{lem:BDOiffBDS}
Let $(A_n)\in\cFn$ and $Y=\ell^p(\Z^2)$. Then $\oplus A_n\in\BDO(Y)$ if and only if $(A_n)\in\BDS$.
\end{lemma}
\begin{Proof}
First note that $\oplus A_n\in\BO(Y)$ iff $\infty>\prop(\oplus A_n)=\sup_n\prop(A_n)$, i.e.~$(A_n)\in\BS$. Then, recalling \eqref{eq:normoplus}, pass to the closure on both sides.
\end{Proof}
\begin{proposition} \label{prop:stable}
For $(A_n)\in\BDS$, the following are equivalent:\\[-7mm]
\begin{enumerate}[label=(\roman*)\,] \itemsep-1mm
\item $(A_n)$ is stable;
\item $(A_n)+\cNn$ is invertible in $\cFn/\cNn$;
\item $\oplus A_n$ (with each $A_n$ extended to $X$ by an according multiple of the identity rather than zero, see (6.4) in \cite{HagLiSei}) is a Fredholm operator;
\item all operators in $\Lim(\oplus A_n)$, i.e., all limit operators of $\oplus A_n$, are invertible;
\item all operators in $\Lay(\Lim(\oplus A_n))$, i.e., all layers of all limit operators of $\oplus A_n$, are invertible. 
\end{enumerate}
\end{proposition}
\begin{Proof}
The equivalence of $(iii)$ and $(iv)$ is by Lemma \ref{lem:BDOiffBDS} and \cite[Thm.\,11]{LiSei:BigQuest}.
The uniform boundedness condition of the inverses in $(v)$ is redundant, see footnote 16 of \cite{HagLiSei}. The rest of the proof is as in \S 6 of \cite{HagLiSei}.
\end{Proof}
Consequently, $\Lim(\oplus A_n)$ and $\Lay(\Lim(\oplus A_n))$ could both act as the set of stability indicators. 
\medskip

\section{Establishing our $\limsup$ formulas in the case $(A_n)\in\BDS$} \label{sec:limsup}
We start in the context that we think is as general as possible: $(A_n)\in\BDS$. The price is that we have to keep the set of stability indicators rather large and unspecific. 
This changes when we specialize to $(A_n)\in\cS$. The following propositions cover both situations.

To this end, given $(A_n)\in\BDS$, put
\begin{equation} \label{eq:BAn}
\cB(A_n)\ :=\ \left\{
\begin{array}{cp{5mm}l}
\Stab(A_n)&&\text{if } (A_n)\in\cS,\\
\Lay(\Lim(\oplus A_n))&&\text{otherwise}
\end{array}
\right.
\end{equation}
with $\Stab(A_n)$ from Definition \ref{def:Stab}.

\begin{remark} \label{rem:defStab}
{\bf a) } Since the original operator $A$ and its approximants $A_n$ all act on $\ell^p(\Z)$ or subspaces thereof, it is desirable to also have the stability indicators of $(A_n)$ in that setting and not acting on $\ell^p(\Z^2)$, even though we pass that space on our way -- hence the $\Lay$ operation.

{\bf b) } The following propositions hold with $\cB(A_n)$ equal to $\Lay(\Lim(\oplus A_n))$. However, it were sufficient to just have the maximizers $B$ of $\|B\|$ and $C$ of $\|C^{-1}\|$ in $\cB(A_n)$, making sure that
\[
\|\cB(A_n)\|_\infty\ =\ \|\Lim(\oplus A_n)\|_\infty
\qquad\text{and}\qquad
\|\big(\cB(A_n)\big)^{-1}\|_\infty\ =\ \|\big(\Lim(\oplus A_n)\big)^{-1}\|_\infty.
\]
For general $(A_n)\in\BDS$, however, these maximizers could be anywhere in $\Lay(\Lim(\oplus A_n))$, whence, without further information on $(A_n)$, we keep $\cB(A_n)$ this large. But for $(A_n)\in\cS$, it is possible to say beforehand in which directions limit operators of $\oplus A_n$ stand a chance of their layers maximizing $\|B\|$ or $\|C^{-1}\|$, while other directions can be disregarded without looking at the particular example. This study is done in Section \ref{sec:domdir}, first for sequences $(A_n)$ of pure finite sections and then for composed finite sections, $(A_n)\in\cS$, both leading to the same economic version of $\cB(A_n)$ termed $\Stab(A_n)$ in Definition \ref{def:Stab} below.
\end{remark}

\subsection{The $\limsup$ of $\|A_n\|$}
\begin{proposition} \label{prop:limsupAn}
For $(A_n)\in\BDS$ and $\cB(A_n)$ from \eqref{eq:BAn}, it holds that
\[
\limsup\|A_n\|\ =\ \max_{B\in\cB(A_n)}\|B\|\ =:\|\cB(A_n)\|_\infty\,.
\]
\end{proposition}

\begin{Proof}
Steps \ci1--\ci3 in the following argument only require $(A_n)\in\BDS$. 
For step~\ci4, let $(A_n)\in\cS\subset\BDS$.
Then, following \S6 of \cite{HagLiSei}, we argue as follows:
\begin{equation} \label{eq:limsupAn}
\limsup \|A_n\|\ \stackrel{\ci1}=\ \|(A_n)+\cN\|_{\cF/\cN}\ \stackrel{\ci2}=\ \|\oplus A_n+\cK\|_{\cL/\cK}\ \stackrel{\ci3}=\ \|\Lim(\oplus A_n)\|_\infty\ \stackrel{\ci4}=\ \|\Stab(A_n)\|_\infty
\end{equation}
\begin{enumerate}[label=\protect\ci{\arabic*}]
\item Technically this is a first semester exercise but conceptually it is an extremely crucial piece of the puzzle that we trace back to \cite{Bo94}.

\item Here $\cL:=L(X)$ and $\cK:=K(X)$ abbreviate the sets of bounded, resp. compact, operators on $X$.
The key is \eqref{eq:normoplus} and that $\oplus A_n$ is compact if and only if $(A_n)$ is a null-sequence. We know this from \cite{DouglasHowe} and \cite{Goro}. See \cite{RaRoSiBook,Li:Book,HagLiSei} for generalisations and extensions. 

\item Lemma \ref{lem:BDOiffBDS} implies that $\oplus A_n\in\BDO$, so that the limit operator approach  \cite{LaRa,RaRoSi1998,RaRoSiBook,Li:Book,CWLi:Memoir,Sei:Survey,LiSei:BigQuest,HagLiSei} applies.
\cite{HagLiSei} shows that $\|.\|_\infty$ is always attained as a maximum. 

\item This step requires $(A_n)\in\cS$, see Remark \ref{rem:defStab} b), Propositions \ref{prop:a-f1} and \ref{prop:a-f2} and \cite[\S6]{HagLiSei}. \qedhere
\end{enumerate}
\end{Proof}

For pure finite sections, $A_n=P_nAP_n$ with $A\in\BDO$, we show in Section \ref{sec:pureimproved} below that $\|A_n\|\to\|A\|$. For composed finite sections $(A_n)\in\cS$, neither do the norms $\|A_n\|$ generally converge, nor is $\limsup\|A_n\|$ generally equal to $\|A\|$. See Example~\ref{ex:kappa} below.

\subsection{The $\limsup$ of $\|A_n^{-1}\|$}
We now come to the sequence of the inverses, $A_n^{-1}$, where, of course, every $A_n$ is inverted as an operator on $\ell^p(-n..n)$, not $\ell^p(\Z)$. 

\begin{proposition} \label{prop:limsupAn-1}
For $(A_n)\in\BDS$ and $\cB(A_n)$ from \eqref{eq:BAn}, it holds that
\[
\limsup\|A_n^{-1}\|\ =\ \max_{B\in\cB(A_n)}\|B^{-1}\|\ 
=:\ \|\big(\cB(A_n)\big)^{-1}\|_\infty\,.
\]
In particular, the $\limsup$ is finite if and only if the maximum is finite, i.e.~$(A_n)$ is stable if and only if every $B\in\cB(A_n)$ is invertible.
\end{proposition}

\begin{Proof}
First, let $(A_n)\in\BDS$ be stable, so that almost all $A_n$ are invertible and the inverses are uniformly bounded. Let us check the following equalities one by one:
\begin{eqnarray} \nonumber
\limsup \|A_n^{-1}\|\ &\stackrel{\ci1}=&\ \|(A_n^{-1})+\cN\|_{\cF/\cN}\ \stackrel{\ci2}=\ \|\oplus A_n^{-1}+\cK\|_{\cL/\cK}\ \stackrel{\ci5}=\ \|R+\cK\|_{\cL/\cK}\\
\ &\stackrel{\ci3}=&\ \|\Lim(R)\|_\infty\ \stackrel{\ci6}=\ \|\big(\Lim(\oplus A_n)\big)^{-1}\|_\infty\ \stackrel{\ci4}=\ \|\big(\Stab(A_n)\big)^{-1}\|_\infty \label{eq:limsupAn-1}
\end{eqnarray}
\begin{enumerate}
\item[\ci1-\ci2] See above. Note that the finitely many non-existent $A_n^{-1}$ can be ignored in the sequence $(A_n^{-1})$ modulo $\cN$ and in the operator $\oplus A_n^{-1}$ modulo $\cK$.

\item[\ci5] To be more precise, put $R:=\oplus R_n$ with $R_n=A_n^{-1}$ where existent and $0$ otherwise. Then $R\,(\oplus A_n)-\oplus P_n\in\cK$ and $(\oplus A_n)R-\oplus P_n\in\cK$, so that $R$ is a Fredholm regularizer of $\oplus A_n$.

\item[\ci3] As above, now with $R\in\BDO$, by Lemma \ref{lem:BDOiffBDS} and \cite[Thm.\,21]{Sei:Survey}.
 Again note that $\|\cdot\|_\infty$ is attained as a maximum, by Theorem~3.2 in \cite{HagLiSei}. 

\item[\ci6] Here we use 
the fact that, by Theorem~16 of \cite{Sei:Survey}, the limit operators of a Fredholm regularizer of $B$ are the inverses of the limit operators of $B$. 

\item[\ci4] This step requires $(A_n)\in\cS$, see Remark \ref{rem:defStab} b), Propositions \ref{prop:a-f1} and \ref{prop:a-f2} and \cite[\S6]{HagLiSei}.
\end{enumerate}
If $(A_n)\in\BDS$ is not stable then, by Proposition \ref{prop:stable}, one operator $B\in\Lim(\oplus A_n)$ is not invertible. If $(A_n)\in\cS$ and $\cB(A_n)=\Stab(A_n)$ then, by Propositions \ref{prop:a-f1} and \ref{prop:a-f2}, the layers of that $B$ are in $\Stab(A_n)$, and one of them is not invertible, by Proposition \ref{prop:stable}.
\end{Proof}

\subsection{The $\limsup$ of pseudospectra, $\speps A_n$ and $\Speps A_n$} \label{sec:Speps}

Here is the translation of Proposition \ref{prop:limsupAn-1} 
into the language of $\mu(\cdot)$ from \eqref{eq:invmu}: If $(A_n)\in\BDS$ then
\begin{equation} \label{eq:muB}
\liminf \mu(A_n)\ =\ \min_{B\in\cB(A_n)} \mu(B).
\end{equation}
Further, notice this standard lemma:
\begin{lemma} \label{lem:Ti}
For arbitrary sets $T_i$ with $i$ in any index set $\cI$, one has
\[
\clos \bigcup_{i\in\cI} \clos T_i\ =\ \clos \bigcup_{i\in\cI} T_i.
\]
\end{lemma}
\begin{Proof}
\framebox{$\subseteq$} 
Start with $T_j\subseteq \cup_{i\in\cI}\, T_i$ for any $j\in\cI$, take the closure on both sides, then the union $\cup_{j\in\cI}$ on the left. Then again take the closure on both sides.
Direction \framebox{$\supseteq$}  is obvious.
\end{Proof}
Then here is our result on the $\limsup$ of $\speps A_n$ and $\Speps A_n$:
\begin{proposition} \label{prop:limsupSpeps}
For a sequence $(A_n)\in\BDS$ and $\cB(A_n)$ from \eqref{eq:BAn}, one has
\begin{eqnarray} \label{eq:limsupspeps}
\limsup\speps A_n &=& \clos\bigcup_{B\in\cB(A_n)}\speps B\ =\  \bigcup_{B\in\cB(A_n)}\clos\speps B
\qquad\text{and}
\\ \label{eq:limsupSpeps}
\limsup\Speps A_n &=& \bigcup_{B\in\cB(A_n)}\Speps B.
\end{eqnarray}
Moreover, all sets in \eqref{eq:limsupspeps} and \eqref{eq:limsupSpeps} are equal.
\end{proposition}
\begin{Proof}
Equality of \eqref{eq:limsupspeps} and \eqref{eq:limsupSpeps} follows from $\clos\speps B=\Speps B$, by Globevnik-Shargorodsky \cite{Globevnik,Shargorodsky08}.
\eqref{eq:limsupSpeps} follows from \eqref{eq:limsupspeps} by passing to the closure under the $\limsup$, which does not change the result, see e.g.~\cite[Prop.\,3.5]{HaRoSi2} in combination with Lemma \ref{lem:Ti}.
So it remains to prove \eqref{eq:limsupspeps}. We start with
\begin{equation} \label{eq:1}
\bigcup_{B\in\cB(A_n)}\speps B\ \subseteq\ \limsup\speps A_n.
\end{equation}
If $\lambda\in\cup_B\, \speps B$ then $\exists C\in\cB(A_n)$ with $\lambda\in\speps C$, i.e.
\[
\eps\ >\ \mu(C-\lambda I)\ \ge\ \min_{B\in\cB(A_n)}\mu(B-\lambda I) \stackrel{P.\,\ref{prop:Stab-lambdaI}}= \min_{D\in\cB(A_n-\lambda I)}\mu(D)\ \stackrel{\eqref{eq:muB}}=\ \liminf \mu(A_n-\lambda I_n),
\]

so that $\mu(A_n-\lambda I_n)<\eps$, i.e.~$\lambda\in\speps A_n$, holds for infinitely many $n\in\N$. So clearly, $\lambda\in\limsup\speps A_n$, and we have \eqref{eq:1}.
Taking the closure on both sides of \eqref{eq:1}, noting that $\limsup$ is already closed \cite[Prop.\,3.2]{HaRoSi2}, gives
\begin{equation} \label{eq:2}
\clos\bigcup_{B\in\cB(A_n)}\speps B\ \subseteq\ \limsup\speps A_n.
\end{equation}
Next we show
\begin{equation} \label{eq:3}
\limsup\speps A_n\ \subseteq\ \bigcup_{B\in\cB(A_n)}\Speps B.
\end{equation}
If $\lambda\in\limsup\speps A_n$ then $\lambda$ is a partial limit, i.e.~$\lambda=\lim \lambda_{n_k}$, of a sequence $(\lambda_n)_{n\in\N}$ with $\lambda_n\in\speps A_n$, 
i.e.~$\mu(A_n-\lambda_nI_n)<\eps$, for all $n\in\N$. By Lipschitz continuity of $\mu$, e.g. \cite[Lemma 2.1]{LiSchmeck:Haus}, it follows that
\[
\mu(A_{n_k}-\lambda I_{n_k})\ \le\ \mu(A_{n_k}-\lambda_{n_k} I_{n_k})+|\lambda-\lambda_{n_k}|\ <\ \eps+|\lambda-\lambda_{n_k}|,
\]
so that
\[
\eps\ \ge\ \liminf\mu(A_{n_k}-\lambda I_{n_k})\ \ge\ \liminf\mu(A_n-\lambda I_n)\ \stackrel{\eqref{eq:muB}}=\ \min_{B\in\cB(A_n)}\mu(B-\lambda I),
\]
whence $\lambda\in\Speps B$ for some $B\in\cB(A_n)$, and we have \eqref{eq:3}.

It remains to puzzle the pieces together, where $(GS)$ is by Globevnik-Shargorodsky \cite{Globevnik,Shargorodsky08}:
\begin{eqnarray*}
\limsup\speps A_n &\stackrel{\eqref{eq:3}}\subseteq& \bigcup_{B\in\cB(A_n)}\Speps B\
\stackrel{(GS)}=\ \bigcup_{B\in\cB(A_n)}\clos\speps B\\
&\subseteq& \clos \bigcup_{B\in\Stab(A_n)}\clos\speps B\ \stackrel{L.\,\ref{lem:Ti}}=\ 
\clos \bigcup_{B\in\cB(A_n)}\speps B\\
&\stackrel{\eqref{eq:2}}\subseteq& \limsup\speps A_n.
\end{eqnarray*}
This shows equality of all sets in this chain of inclusions and hence proves \eqref{eq:limsupspeps}.
\end{Proof}

Whether and when the two $\limsup$ in Proposition \ref{prop:limsupSpeps} are proper (Hausdorff) limits is also answered in Section~\ref{sec:limsup=lim}.

\subsection{The $\limsup$ of the condition numbers, $\kappa(A_n)=\|A_n\|\cdot\|A_n^{-1}\|$}
The situation is less satisfactory for the condition numbers. Their $\limsup$ is between the largest product, $\|B\|\|B^{-1}\|$, and the product of the largest factors, $\|B\|$ and $\|C^{-1}\|$, with $B,C\in\cB(A_n)$. 
\begin{proposition} \label{prop:limsupkappa}
For $(A_n)\in\BDS$ with $\cB(A_n)$ from \eqref{eq:BAn}, one has
\begin{eqnarray*}
\sup_{B\in\cB(A_n)}\kappa(B)\ \le\ \limsup \kappa(A_n)&\le& \limsup \|A_n\|\cdot\limsup\|A_n^{-1}\|\\[-3mm]
&=&\|\cB(A_n)\|_\infty\cdot\|\big(\cB(A_n)\big)^{-1}\|_\infty.
\end{eqnarray*}
\end{proposition}
\begin{proof}[\bf Sketch of proof.]
For instable $(A_n)$, the statement is $\infty\le\infty\le\infty$; for stable $(A_n)$, it follows by Polski's theorem \cite[Prop.\,2.2]{BoSi2} and Lemma~\ref{lem:liminf}, bounding $\|B\|$ and $\|B^{-1}\|$ from above.
\end{proof}
An equality for $\limsup\kappa(A_n)$ in case $(A_n)\in\cS$ is derived in Proposition \ref{prop:conv} below.
The inequalities in Proposition \ref{prop:limsupkappa} can be strict, see Example \ref{ex:kappa}.
We do not know whether $\sup\kappa(B)$ is also a maximum. When exactly $\limsup=\lim$ holds for any of the three occurrences in Proposition \ref{prop:limsupkappa} is the topic of Section \ref{sec:limsup=lim}. If one of $\limsup \|A_n\|$ and $\limsup\|A_n^{-1}\|$ is a limit then the second ``$\le$'' sign is an equality. If even both are limits then also $\limsup\kappa(A_n)$ is a limit.

\begin{example} \label{ex:kappa}
{\bf a) } Let $A=BC$ and $A_n=(P_nBP_n)(P_nCP_n)$ for $n\in\N$ with
\[
B=\diag(\dots,D,D,\boxed 1,D,D,\dots)
\qquad\text{and}\qquad 
C=\diag(\dots, E, E,\boxed 1,E,E, \dots),
\]
where 
$D={2~~1\choose 0~\sfrac12}$, $E={2~~0\choose -1~\sfrac12}$ and $B_{00}=\boxed 1=C_{00}$.
Then, for $k\in\N$,
\[
A_{2k}=\diag(DE,\dots DE,\boxed 1, DE,\dots,DE)
\quad\text{and}\quad
A_{2k+1}=\diag(\sfrac14,A_{2k},4).
\]
Both $DE={~3~~~\sfrac12\choose -\sfrac12~\sfrac14}$ and $(DE)^{-1}={\sfrac14~-\sfrac12\choose \sfrac12~~~3}$ have norm\footnote{Meaning the induced operator norm for $p=2$. The tedious computation of the largest singular value can be replaced by noting that $\|DE\|_F=\|(DE)^{-1}\|_F=\sqrt{9+\frac14+\frac14+\frac1{16}}=\frac14\sqrt{153}\in(3,\frac{13}4)$ for the Frobenius norm and recalling that $\frac 1{\sqrt2}\|M\|_F\le\|M\|\le\|M\|_F$ for $2\times 2$ matrices $M$, so that $2<\frac3{\sqrt 2}<\|DE\|=\|(DE)^{-1}\|<\frac{13}4<4$.} $N:=\frac18(11+\sqrt{185})\approx 3.075$. So
\[
\begin{array}{rp{3mm}cp{3mm}cp{5mm}l}
\|A_n\|=N,&&\|A_n^{-1}\|=N,&&\kappa(A_n)=N^2&&\text{if $n\ge 2$ is even}\\
\text{and}\qquad\|A_n\|=4,&&\|A_n^{-1}\|=4,&&\kappa(A_n)=16&&\text{if $n\ge 3$ is odd},\\
\end{array}
\]
whence $\limsup\kappa(A_n)=16$. The set $\cB(A):=\Stab(A_n)$ consists of $A,F,G,H,J$ with
\[
F=\diag(DE,\dots),\quad G=\diag(\dots,DE),\quad H=\diag(\sfrac14, DE,\dots),\quad J=\diag(\dots,DE,4),
\]
so that the first ``$\le$'' sign in Proposition \ref{prop:limsupkappa} is ``$<$'' since $\|J\|=4$ maximizes the norm, $\|H^{-1}\|=4$ maximizes the norm of the inverse and $\kappa(H)=\kappa(J)=4N$ maximize the condition number.

{\bf b) } Instead, putting $A_n=(P_nBP_n)(P_nCP_n)$ with
\[
B=\diag(\dots,D,D,D,\dots)
\qquad\text{and}\qquad 
C=\diag(\dots, E, E,E, \dots),
\]
both $A_{2k}$ and $A_{2k+1}$ cut through a block on the left or right endpoint, leading to
\[
\begin{array}{rp{3mm}cp{3mm}cp{5mm}l}
\|A_n\|=4,&&\|A_n^{-1}\|=N,&&\kappa(A_n)=4N&&\text{if $n\ge 2$ is even}\\
\text{and}\qquad\|A_n\|=N,&&\|A_n^{-1}\|=4,&&\kappa(A_n)=4N&&\text{if $n\ge 3$ is odd},\\
\end{array}
\]
so that $\|A_n\|$ and $\|A_n^{-1}\|$ are both alternating (oppositely) between $4$ and $N$.
Even though $\limsup\kappa(A_n)=4N$ is a limit now, the second ``$\le$'' sign in Proposition \ref{prop:limsupkappa} is a ``$<$'' here.
\end{example}

\subsection{Pure finite sections: improved asymptotic results of $\|A_n\|$ and $\kappa(A_n)$} \label{sec:pureimproved}
First a standard result that is sometimes \cite{BoSi2,HaRoSi2,BoGru} stated as an add-on to Banach-Steinhaus:
\begin{lemma} \label{lem:liminf}
If $A\in L(X)$ and $\|A_nx\|\to\|Ax\|$ for all $x\in X$ then $\|A\|\le\liminf\|A_n\|$.
\end{lemma}
~\\[-17mm]
\begin{Proof}
Let $\eps>0$ and $x\in X$ with $\|x\|=1$ and $\|Ax\|\stackrel{\eps/2}\approx \|A\|$, where $a \stackrel{\delta}\approx b$ means $|a-b|<\delta$. For sufficiently large $n\in\N$, by assumption, $\|A_n x\|\stackrel{\eps/2}\approx\|Ax\|$, so that $\|A\|\stackrel{\eps}\approx\|A_n x\|\le\|A_n\|$.\qedhere
\end{Proof}
As a consequence, for pure finite sections, we can improve Propositions \ref{prop:limsupAn} and \ref{prop:limsupkappa} as follows:
\begin{proposition} \label{prop:limAn}
For the pure finite sections, $A_n=P_n AP_n$, of an operator $A\in\BDO$, 
\begin{eqnarray*}
\lim\|A_n\| &=& \|A\|,\\
\limsup\|A_n^{-1}\| &=& \|\big(\Stab(A_n)\big)^{-1}\|_\infty
\qquad\text{and}\\
\limsup\kappa(A_n) &=& \|A\|\cdot \|\big(\Stab(A_n)\big)^{-1}\|_\infty\,.
\end{eqnarray*}
\end{proposition}
\begin{Proof} By $A_n\to A$ and Lemma \ref{lem:liminf}, it follows that $\|A\|\le\liminf\|A_n\|$.
Morover, by $\|P_n\|=1$, we have $\|A_n\|= \|P_nAP_n\|\le\|A\|$. It follows that $\|A\|\le\liminf\|A_n\|\le\limsup\|A_n\|\le\|A\|$, so that $\lim\|A_n\|$ exists and equals $\|A\|$.
The rest is by Proposition \ref{prop:limsupAn-1} and $\kappa(A_n)=\|A_n\|\|A_n^{-1}\|$.\qedhere
\end{Proof}
Note that $\limsup\kappa(A_n)$ is now given by an equality and that it is a limit if and only if $\limsup\|A_n^{-1}\|$ is a limit. When exactly that is the case is answered in Section~\ref{sec:limsup=lim} below.

\section{Dominant directions of $\oplus A_n$ and the definition of $\Stab(A_n)$} \label{sec:domdir}

The propositions in Section \ref{sec:limsup} ask for the limit operators of $\oplus A_n$. Let us study this set and identify its maximal elements in terms of $\|B\|$ and $\|C^{-1}\|$. Our analysis of the latter is limited to the algebra $\cS$. We start with pure finite sections and then proceed to the composed case.
\begin{definition} \label{def:dom}
Given two operators, $A$ and $B$, we will say that $A$ {\sl dominates} $B$ (or that $B$ {\sl is dominated by} $A$) if $\|A\|\ge\|B\|$ and $\|A^{-1}\|\ge\|B^{-1}\|$, while keeping the convention of putting $\|A^{-1}\|=\infty$ in case of non-invertibility.
\end{definition}
In particular, if $A$ dominates $B$ then also invertibility of $A$ implies that of $B$. So if, among all $B_n\in \Lay(\Lim(\oplus A_n))$, we ignore the $B_n$ that are dominated by others from the set, the remaining set still captures stability of $(A_n)$ and the $\limsup$'s of Section \ref{sec:limsup}.

\subsection{First: pure finite sections, $A_n=P_nAP_n$}
For $A_n=P_nAP_n$ with $A\in\BDO$, shift $\oplus A_n$ along a sequence $h=(h_k)=(\,{i_k\choose j_k}\,)$ in $\Z^2$:
\begin{eqnarray} \nonumber
(\oplus_n A_n)_h &\ot&
S_{-{i_k\choose j_k}}(\oplus_n A_n)S_{{i_k\choose j_k}} \ =\ \oplus_n\, S_{-i_k} A_{n+j_k} S_{i_k}\ =\ \oplus_n\, S_{-i_k} P_{n+j_k} A P_{n+j_k} S_{i_k}\\ 
&=&\oplus_n\, (S_{-i_k} P_{n+j_k} S_{i_k}) (S_{-i_k} A S_{i_k}) (S_{-i_k} P_{n+j_k} S_{i_k})\ \to\ \oplus_n\, B_n \label{eq:oplusBn}
\end{eqnarray}
as $k\to\infty$ if the pointwise limit exists. By uniqueness of the limit, $(\oplus A_n)_h=\oplus B_n$.
\begin{proposition} \label{prop:a-f1}
Let $(A_n)$ be the sequence of pure finite sections, $A_n=P_nAP_n$, of an operator $A\in\BDO$ and let $\oplus B_n$ be a limit operator of $\oplus A_n$ with respect to a sequence $h$ in $\Z^2$.
\begin{abc}
\item
Depending on the direction of $h$ in $\Z^2$, the layers $B_n$ are of the form
\begin{itemize}
\itemsep0mm
\item[(a)]\ 
\begin{tikzpicture}[scale=0.3]
\fill[gray!20] (0,0) -- (1,1) -- (-1,1);
\draw[-latex,line width=0.15mm] (0,0) -- (0,1);
\end{tikzpicture}
\quad the operator $A$ itself,

\item[(b)]\ 
\begin{tikzpicture}[scale=0.3]
\fill[gray!20] (0,0) -- (1,1) -- (-1,1);
\draw[-latex,line width=0.15mm] (0,0) -- (0.4,1);
\end{tikzpicture}
\quad a limit operator of $A$,

\item[(c)]\ 
\begin{tikzpicture}[scale=0.3]
\fill[gray!20] (0,0) -- (1,1) -- (-1,1);
\draw[-latex,line width=0.15mm] (0,0) -- (1,1);
\end{tikzpicture}
\quad $P_-A_gP_-$ with $A_g\in\Lim_+(A)$,

\item[(d)]\ 
\begin{tikzpicture}[scale=0.3]
\fill[gray!20] (0,0) -- (1,1) -- (-1,1);
\draw[-latex,line width=0.15mm] (0,0) -- (-1,1);
\end{tikzpicture}
\quad $P_+A_gP_+$ with $A_g\in\Lim_-(A)$,

\item[(e)]\ 
\begin{tikzpicture}[scale=0.3]
\fill[gray!20] (0,0) -- (1,1) -- (-1,1);
\draw[-latex,line width=0.15mm] (0,0) -- (1,0.4);
\end{tikzpicture}
\quad $0$\quad or

\item[(f)]\ translates, $B_n=S_{-c} B'_n S_c$, of any of the operators $B'_n$ in (a)--(e).
\end{itemize}

\item
Case (b) is dominated by (a) and case (f) is dominated by the previous cases. 
\end{abc}
\end{proposition}

\begin{remark} \label{rem:a-f}
{\bf a) }
Since every $A_n$ is only of interest as an operator on $\ell^p(-n..n)$, we can ignore the case (e) and we have to understand the operators (c) as operators on $\im P_-=\ell^p(..0)$ and the operators (d) as operators on $\im P_+=\ell^p(0..)$.
Alternatively, extend every $A_n$ to $\ell^p(\Z)$ by $c$ times the identity before forming $\oplus A_n$. Choose $c\in\R$ such that it does not change the property of interest, e.g. put $c=0$ in \eqref{eq:limsupAn} and $c>\sup\|A_n\|$ in \eqref{eq:limsupAn-1}. This way was followed in \cite[\S6]{HagLiSei}. In the current paper we ignore the $n$th layer outside of $-n..n$ whence, in the limit, we have to be a bit
flexible and tolerant about having operators on $\ell^p(\Z), \ell^p(0..)$ and $\ell^p(..0)$ in the same set.

{\bf b) }
It is tempting to classify the cases (a)--(f) via the angle\footnote{Indeed, by compactness of the unit circle, $h_k/|h_k|$ has a convergent subsequence that we can pass to without changing the limit operator $(\oplus A_n)_h$, hence giving $h$ an angle, asymptotically.} that $h$ asymptotically encloses with the $n$-axis: (a) is for $0^\circ$,  (b) for $(0^\circ,45^\circ)$, (c) and (d) are for $\pm 45^\circ$, respectively, (e) is for $\not\in [-45^\circ,45^\circ]$, and (f) is a finite shift of a sequence in one of (a)--(e).
While this is roughly what happens, it is incorrect since (b) and (e) also reach $0^\circ$ and $\pm 45^\circ$. E.g., $h_k={k\choose k^2}$ has angle $0^\circ$ with the $n$-axis but is in case (b), not (a)+(f), or ${k^2\pm k\choose k^2}$ has angle $45^\circ$ and is (b), resp.~(e), not (c)+(f).
So we stick with the following clumsy way of distinguishing (a)--(f). 
\end{remark}

\begin{proof}[\bf Proof of Proposition \ref{prop:a-f1}]
{\bf a) } Let $A_n=P_nAP_n$ and let the sequence $h=(h_k)=(\,{i_k\choose j_k}\,)$ in $\Z^2$ be such that $|h_k|\to\infty$ and that the limit operator $(\oplus A_n)_h=:\oplus B_n$ exists.
By \eqref{eq:oplusBn}, for each $n\in\Z$,
\begin{equation} \label{eq:B0}
B_n\ \ot\  (S_{-i_k} P_{n+j_k} S_{i_k}) (S_{-i_k} A S_{i_k}) (S_{-i_k} P_{n+j_k} S_{i_k})
\quad\text{as}\quad k\to\infty.
\end{equation}
Without loss, we can focus on $B_0$ since every other layer, say $B_m$, is the $0$th layer, $C_0$, of
\[
\oplus C_n\ =\ \oplus B_{n+m}\ =\ S_{-{0\choose m}}(\oplus B_n)S_{0\choose m}\ =\ (\oplus A_n)_{h+{0\choose m}},
\]
which is just another limit operator of $\oplus A_n$.
\begin{itemize}
\itemsep0mm
\item[(a)]\ If $h$ is a subsequence of $\N\cdot{0\choose 1}$, i.e.~$h_k={0\choose j_k}$ with $j_k\to+\infty$ then $S_{-i_k}=I=S_{i_k}$, so that, by \eqref{eq:B0} with $n=0$,
$B_0\ot P_{j_k} A P_{j_k} \to A$ as $k\to\infty$, whence $B_0=A$.

\item[(b)]\ If $|i_k|\to\infty$ but $|i_k|-j_k\to-\infty$, so that
$i_k-j_k\to-\infty$ and $-i_k-j_k\to-\infty$ then
\[
S_{-i_k} P_{j_k} S_{i_k}\ =\ P_{-j_k-i_k..j_k-i_k}\ \to\ I
\quad\text{as}\quad k\to\infty,
\quad\text{so that, by \eqref{eq:B0},}\quad
B_0=A_i,
\]
the limit operator of $A$ with respect to the sequence $i=(i_k)_k$. 

\item[(c)]\ If $h$ is a subsequence of $\N\cdot{1\choose 1}$, i.e.~$h_k={j_k\choose j_k}$ with $j_k\to+\infty$ then
\[
S_{-i_k} P_{j_k} S_{i_k}\ =\ P_{-2j_k..0}\ \to\ P_{..0}\ =\ P_-
\ \text{as}\ k\to\infty,
\quad\text{so that, by \eqref{eq:B0},}\quad
B_0=P_-A_jP_-
\]
with $A_j$ denoting the limit operator of $A$ w.r.t.~the sequence $j=(j_k)_k\to+\infty$.

\item[(d)]\ If $h$ is a subsequence of $\N\cdot{-1\choose 1}$, i.e.~$h_k={-j_k\choose j_k}$ with $j_k\to+\infty$ then
\[
S_{-i_k} P_{j_k} S_{i_k}\ =\ P_{0..2j_k}\ \to\ P_{0..}\ =\ P_+
\ \text{as}\  k\to\infty,
\quad\text{so that, by \eqref{eq:B0},}\quad
B_0=P_+A_{-j}P_+
\]
with  $A_{-j}$ denoting the limit operator of $A$ w.r.t.~the sequence $-j=(-j_k)\to-\infty$. 

\item[(e)]\ If $|i_k|-j_k\to+\infty$, i.e. $i_k-j_k\to +\infty$ or $-i_k-j_k\to+\infty$ then
\[
S_{-i_k} P_{j_k} S_{i_k}\ =\ P_{-j_k-i_k..j_k-i_k}\ \to\ 0
\quad\text{as}\quad k\to\infty,
\quad\text{so that}\quad
B_0=0.
\]

\item[(f)]\ If $h$ differs from one of (a)--(e) by a sequence $g$ with a bounded subsequence then pass to a subsequence of $h$, where $g$ is constant, say $g\equiv c$, (Bolzano-Weierstra\ss\ in $\Z^2$), not changing the limit operator, $(\oplus A_n)_h=\oplus B_n$.
So $B_0=S_{-c}B'_0S_c$ with $B'_0$ from one of (a)--(e).
\end{itemize}

{\bf b) } For (b), it remains to note that $\|A\|\ge\|A_g\|$ for every limit operator $A_g$ of $A$ and that $(A_g)^{-1}=(A^{-1})_g$ if $A$ is invertible, see e.g.~\cite{RaRoSi1998}, whence also $\|A^{-1}\|\ge\|(A^{-1})_g\|=\|(A_g)^{-1}\|$.

For (f), note that $\|S_{-c}BS_c\|=\|B\|$ and $\|(S_{-c}BS_c)^{-1}\|=\|S_{-c}B^{-1}S_c\|=\|B^{-1}\|$.
\end{proof}

After Proposition \ref{prop:a-f1} we are left with the layers $B_n$ of type (a), (c) and (d), and this is what we will collect in $\Stab(A_n)$.
In the case of pure finite sections, $A_n=P_nAP_n$, checking these three directions is enough, by Proposition \ref{prop:a-f1}, to find, among $\Lay(\Lim(\oplus A_n))$,
\begin{itemize} \itemsep-1mm
\item a maximizer of $\|B_n\|$,
\item a maximizer of $\|B_n^{-1}\|$ and
\item a subset of operators $B_n$ whose invertibility implies that of all the others.
\end{itemize}
In Proposition \ref{prop:a-f2} below we show that the same three directions, ${0\choose 1}$, ${1\choose 1}$ and ${-1\choose 1}$, are sufficient for the same purpose also for composed finite sections, $(A_n)\in\cS$. Wrapping up, we evaluate
\begin{equation} \label{eq:B0gen}
\oplus B_n \ot S_{-{i_k\choose j_k}}(\oplus A_n)S_{{i_k\choose j_k}}  = \oplus\, S_{-i_k} A_{n+j_k} S_{i_k},
\  \text{ so that }\  B_0 \ot  S_{-i_k} A_{j_k} S_{i_k}\text{ as }k\to\infty,
\end{equation}
for $j_k\to+\infty$ and (a) $i_k=0$, (c) $i_k=j_k$ and (d) $i_k=-j_k$.
So here is the definition:
\begin{definition} \label{def:Stab}
For a sequence $(A_n)\in\cS$, let $\,\Stab(A_n)$ denote the set consisting of the pointwise limits of all subsequences of
\[
A_n,\qquad S_{-n}A_nS_n\qquad\text{and}\qquad S_nA_nS_{-n}.
\]
We call the elements of $\Stab(A_n)$ the {\sl stability indicators} of $(A_n)$.
\end{definition}
\begin{remark}
{\bf a) } Every sequence $(A_n)\in\cS$ converges pointwise. So, for the first sequence in Definition \ref{def:Stab}, there is actually no need to pass to subsequences. The pointwise limit is the operator $A$ that was to be approximated in the first place. In particular, $A\in\Stab(A_n)$ holds.

{\bf b) } In cases, where also $S_{-n}A_nS_n$ and $S_nA_nS_{-n}$ converge, e.g.~for $(A_n)\in\clos\alg\{P_nLP_n:L \text{ banded Laurent, see Example \ref{ex:Laurent}}\}$, $\Stab(A_n)$ consists of just three elements and all $\limsup$'s in Section~\ref{sec:limsup} are limits, by the results of Section~\ref{sec:limsup=lim} below.

{\bf c) } That at least $\|A_n\|$ converges in situation b) can also be shown by applying Lemma~\ref{lem:liminf} to the limits $A,\, B,\, C$ of $A_n,\ S_{-n}A_nS_n$ and $S_nA_nS_{-n}$, using that all $S_k$ are isometries. Indeed,
\[
\|A\|,\|B\|, \|C\| \stackrel{L.\,\ref{lem:liminf}}\le \liminf\|A_n\|\ \le\ \limsup\|A_n\| \stackrel{P.\,\ref{prop:limsupAn}}= \|\Stab(A_n)\|_\infty\ =\ \max\{\|A\|,\|B\|, \|C\|\},
\]
whence $\liminf\|A_n\|= \limsup\|A_n\|=\lim\|A_n\|= \max\{\|A\|,\|B\|, \|C\|\}$.
\end{remark}
\begin{proposition} \label{prop:Stabpure}
For pure finite sections, $A_n=P_nAP_n$ with $A\in\BDO$, we have
\[
\Stab(A_n)\ =\ \Big\{\ A,\ \ P_-A_fP_-,\ \ P_+A_gP_+\ \ :\ \ A_f\in\Lim_+(A), \ A_g\in\Lim_-(A)\ \Big\}.
\]
\end{proposition}
\begin{proposition} \label{prop:Stab-lambdaI}
In either case,\ \ {\bf a)\,} $(A_n)\in\cS$ and $\cB(A_n):=\Stab(A_n)$ or\\
{\bf b)\,} $(A_n)\in\BDS$ and $\cB(A_n):=\Lay(\Lim(\oplus A_n))$, we have
\[
\cB(A_n-\lambda I_n)\ =\ \cB(A_n)-\lambda I\ :=\ \{B-\lambda I:B\in\cB(A_n)\},\qquad \lambda\in\C.
\]
\end{proposition}
\begin{remark} \label{rem:whichI}
Recall that $B\in\cB(A_n)$ can be an operator on $\ell^p(\Z)$ (bi-infinite matrix) or $\ell^p(k..)$ or $\ell^p(..k)$ with $k\in\Z$ (different semi-infinite matrices). When writing $B-\lambda I$ then $I$ shall denote the identity on the corresponding space.
\end{remark}
\begin{proof}[\bf Proof of Proposition \ref{prop:Stab-lambdaI}]
{\bf a) } By Definition \ref{def:Stab}, $\Stab(A_n-\lambda I_n)$ consists of the pointwise limits of subsequences of
\[
A_n-\lambda I_n,\quad S_{-n}(A_n-\lambda I_n)S_n
\qquad\text{and}\qquad S_n(A_n-\lambda I_n)S_{-n}
\]
as $n\to\infty$.
The conclusion follows from $I_n\to I$, $S_{-n}I_nS_n\to P_-$ and $S_nI_nS_{-n}\to P_+$.

{\bf b) } 
Every $\oplus B_n\in\Lim(\oplus(A_n-\lambda I_n))=\Lim(\oplus A_n-\lambda\oplus I_n)$ is of the form $\oplus C_n-\lambda\oplus D_n$, where $\oplus C_n\in\Lim(\oplus A_n)$ and $D_n=I$ or $P_{k-n..}$ or $P_{..k+n}$ with $k\in\Z$. After decomposition in layers, $B_n=C_n-\lambda D_n$, the claim follows.
\end{proof}

\subsection{Now: composed finite sections, $(A_n)\in\cS$}
Let us check that still ${0\choose 1}$, ${1\choose 1}$ and ${-1\choose 1}$ are the dominant directions in $\oplus A_n$ when $(A_n)\in\cS$, say \eqref{eq:Analg} holds. Shifting $\oplus A_n$ along the integer sequence $h=(h_l)=(\,{\alpha_l\choose \beta_l}\,)$ in $\Z^2$, leads to
\begin{eqnarray} \nonumber
(\oplus_n A_n)_h &\ot&
S_{-{\alpha_l\choose \beta_l}}(\oplus_n A_n)S_{{\alpha_l\choose\beta_l}} \ =\ \oplus_n\, S_{-\alpha_l} A_{n+\beta_l} S_{\alpha_l}\\ \nonumber
 &=& \oplus_n\, S_{-\alpha_l} \Big(\lim_i\sum_j\prod_k (P_{n+\beta_l}A^{(i,j,k)}P_{n+\beta_l})\Big) S_{\alpha_l}\\ \label{eq:oplusBn2}
&=&\oplus_n\,\lim_i\sum_j\prod_k (S_{-\alpha_l} P_{n+\beta_l} S_{\alpha_l}) (S_{-\alpha_l} A^{(i,j,k)} S_{\alpha_l}) (S_{-\alpha_l} P_{n+\beta_l} S_{\alpha_l})\ \to\ \oplus_n\, B_n\qquad
\end{eqnarray}
as $l\to\infty$ if the limits exist.
With very much analogy to Proposition \ref{prop:a-f1}, we have:
\begin{proposition} \label{prop:a-f2}
For $(A_n)\in\cS$,
let $\oplus B_n$ be a limit operator of $\oplus A_n$ w.r.t.~a sequence $h$ in $\Z^2$.
\begin{abc}
\item
Depending on the direction of $h$ in $\Z^2$, the operators $B_n$ are of the form
\begin{itemize}
\itemsep0mm
\item[(a)]\ 
\begin{tikzpicture}[scale=0.3]
\fill[gray!20] (0,0) -- (1,1) -- (-1,1);
\draw[-latex,line width=0.15mm] (0,0) -- (0,1);
\end{tikzpicture}
\quad the pointwise limit of $A_n$, that is $\lim_{i\to\infty}\sum_{j\in J_i}\prod_{k\in K_i} A^{(i,j,k)}=: A$,

\item[(b)]\ 
\begin{tikzpicture}[scale=0.3]
\fill[gray!20] (0,0) -- (1,1) -- (-1,1);
\draw[-latex,line width=0.15mm] (0,0) -- (0.4,1);
\end{tikzpicture}
\quad a limit operator of $A$,

\item[(c)]\ 
\begin{tikzpicture}[scale=0.3]
\fill[gray!20] (0,0) -- (1,1) -- (-1,1);
\draw[-latex,line width=0.15mm] (0,0) -- (1,1);
\end{tikzpicture}
\quad $\lim_{i\to\infty}\sum_{j\in J_i}\prod_{k\in K_i} (P_-A_g^{(i,j,k)}P_-)$ with $A_g^{(i,j,k)}\in\Lim_+(A^{(i,j,k)})$,

\item[(d)]\ 
\begin{tikzpicture}[scale=0.3]
\fill[gray!20] (0,0) -- (1,1) -- (-1,1);
\draw[-latex,line width=0.15mm] (0,0) -- (-1,1);
\end{tikzpicture}
\quad $\lim_{i\to\infty}\sum_{j\in J_i}\prod_{k\in K_i} (P_+A_g^{(i,j,k)}P_+)$ with $A_g^{(i,j,k)}\in\Lim_-(A^{(i,j,k)})$m

\item[(e)]\ 
\begin{tikzpicture}[scale=0.3]
\fill[gray!20] (0,0) -- (1,1) -- (-1,1);
\draw[-latex,line width=0.15mm] (0,0) -- (1,0.4);
\end{tikzpicture}
\quad $0$\quad or

\item[(f)]\ translates, $B_n=S_{-c} B'_n S_c$, of any of the operators $B'_n$ in (a)--(e).
\end{itemize}

\item
Case (b) is dominated by (a) and case (f) is dominated by the previous cases. 
\end{abc}
\end{proposition}

\noindent
Again, Remark \ref{rem:a-f} applies.

\begin{Proof}
\quad {\bf a) } Let the sequence $h=(h_l)=(\,{\alpha_l\choose \beta_l}\,)$ in $\Z^2$ be such that $|h_l|\to\infty$ and that the limit operator $(\oplus A_n)_h=:\oplus B_n$ exists.
As in the proof of Proposition \ref{prop:a-f1}, w.l.o.g., let $n=0$. By \eqref{eq:oplusBn2},
\begin{equation} \label{eq:B02}
B_0\ \ot\  \lim_{i\to\infty}\sum_{j\in J_i}\prod_{k\in K_i} (S_{-\alpha_l} P_{\beta_l} S_{\alpha_l}) (S_{-\alpha_l} A^{(i,j,k)} S_{\alpha_l}) (S_{-\alpha_l} P_{\beta_l} S_{\alpha_l})
\quad\text{as}\quad l\to\infty.
\end{equation}
\begin{itemize}
\itemsep0mm
\item[(a)]\ If $h$ is a subsequence of $\N\cdot{0\choose 1}$, i.e.~$h_l={0\choose \beta_l}$ with $\beta_l\to+\infty$ then $S_{-\alpha_l}=I=S_{\alpha_l}$, so that, by \eqref{eq:B02},
$B_0\ot \lim_i\sum_j\prod_k P_{\beta_l}A^{(i,j,k)}P_{\beta_l} \to A$ as $l\to\infty$, whence $B_0=A$.

\item[(b)]\ If $|\alpha_l|\to\infty$ but $|\alpha_l|-\beta_l\to-\infty$, so that
$\alpha_l-\beta_l\to-\infty$ and $-\alpha_l-\beta_l\to-\infty$ then
\[
S_{-\alpha_l} P_{\beta_l} S_{\alpha_l}\ =\ P_{-\beta_l-\alpha_l..\beta_l-\alpha_l}\ \to\ I
\quad\text{as}\quad l\to\infty,
\]
so that, by \eqref{eq:B02} and by the standard properties of limit operators, e.g.~\cite[Prop.\,3.4]{Li:Book},
\[
B_0\ =\ \lim_{i\to\infty}\sum_{j\in J_i}\prod_{k\in K_i} A_\alpha^{(i,j,k)}
\ =\ \Big(\lim_{i\to\infty}\sum_{j\in J_i}\prod_{k\in K_i} A^{(i,j,k)}\Big)_\alpha\ =\ A_\alpha,
\]
meaning the limit operator of $A$ with respect to the sequence $\alpha=(\alpha_l)_l$ in $\Z$. 

\item[(c)]\ If $h$ is a subsequence of $\N\cdot{1\choose 1}$, i.e.~$h_l={\beta_l\choose \beta_l}$ with $\beta_l\to+\infty$ then
\[
S_{-\alpha_l} P_{\beta_l} S_{\alpha_l}\ =\ P_{-2\beta_l..0}\ \to\ P_{..0}\ =\ P_-
\quad\text{as}\quad l\to\infty,
\]
so that, by \eqref{eq:B02},
\[
B_0\ =\ \lim_{i\to\infty}\sum_{j\in J_i}\prod_{k\in K_i}(P_- A_\beta^{(i,j,k)}P_-)
\]
with $\beta=(\beta_l)_l\to+\infty$.

\item[(d)]\ This is again in analogy to (c).

\item[(e, f)]\ Both are as in Proposition \ref{prop:a-f1}.
\end{itemize}

\noindent
{\bf b) } See the proof of Proposition \ref{prop:a-f1}. \qedhere
\end{Proof}
... again leaving us with (a), (c) and (d) and confirming our Definition \ref{def:Stab} of $\Stab(A_n)$.
\begin{proposition} \label{prop:Stabalg}
If $(A_n)$ is in the finite section algebra $\cS$, say \eqref{eq:(An)alg} holds, we have
\begin{eqnarray*}
\Stab(A_n) \ =\ \Big\{\ A,&& \lim_{i\to\infty}\sum_{j\in J_i}\prod_{k\in K_i} (P_-A_f^{(i,j,k)}P_-),\quad \lim_{i\to\infty}\sum_{j\in J_i}\prod_{k\in K_i} (P_+A_g^{(i,j,k)}P_+)\\
&& \qquad:\quad A_f^{(i,j,k)}\in\Lim_+(A^{(i,j,k)}), \quad A_g^{(i,j,k)}\in\Lim_-(A^{(i,j,k)})\ \Big\}
\end{eqnarray*}
with $A:=\lim_{i\to\infty}\sum_{j\in J_i}\prod_{k\in K_i} A^{(i,j,k)}$ denoting the pointwise limit of $A_n$.
\end{proposition}

\section{Subsequence versions of our $\limsup$ formulas} \label{sec:subseq}
As a preparation for Section \ref{sec:limsup=lim}, we prove subsequence versions of Propositions \ref{prop:limsupAn} and \ref{prop:limsupAn-1}, implying subsequence versions of Propositions \ref{prop:limsupSpeps} and \ref{prop:limsupkappa}. 
The goal is to show, for $(A_n)\in\BDS$,
\begin{eqnarray} 
\label{eq:limsupAhn}
\limsup\|A_{h_n}\| &=& \|\cB_h(A_n)\|_\infty\,,\\
\label{eq:limsupAhn-1}
\limsup\|A_{h_n}^{-1}\| &=& \|\big(\cB_h(A_n)\big)^{-1}\|_\infty\,,\\
\label{eq:limsupkappaAhn}
\sup_{B\in\cB_h(A_n)}\kappa(B)\ \le\ 
\limsup\kappa(A_{h_n}) &\le& \|\cB_h(A_n)\|_\infty\cdot\|\big(\cB_h(A_n)\big)^{-1}\|_\infty\,,\\[-3mm]
\label{eq:limsupSpepsAhn}
\limsup\Speps A_{h_n} &=& \bigcup_{B\in\cB_h(A_n)}\Speps B
\end{eqnarray}
for arbitrary monotonic sequences $h=(h_n)$ in $\N$. 

For $(A_n)\in\BDS$ and a monotonic sequence $h=(h_n)$ in $\N$, put
\begin{equation} \label{eq:B_hAn}
\cB_h(A_n)\ :=\ \left\{
\begin{array}{cp{5mm}l}
\Stab_h(A_n)&&\text{if } (A_n)\in\cS,\\
\Lay(\Lim_{{*\choose h}}(\oplus A_n))&&\text{otherwise,}
\end{array}
\right.
\end{equation}
where $\Stab_h(A_n)$ is given by Definition \ref{def:Stab_h} below and $\Lim_{{*\choose h}}(\oplus A_n)$ is the set of all limit operators $(\oplus A_n)_g$ with sequences $g=({i_k\choose j_k})_k$ in $\Z^2$ with $i_k\in\Z$ and $j=(j_k)$ a subsequence of $h$.
\begin{definition} \label{def:Stab_h}
For a sequence $(A_n)\in\cS$ and a monotonic sequence $h=(h_n)$ in $\N$, let $\Stab_h(A_n)$ denote the set consisting of the pointwise limits of all subsequences of
\[
A_{h_n},\qquad S_{-h_n}A_{h_n}S_{h_n}\qquad\text{and}\qquad S_{h_n}A_{h_n}S_{-h_n}.
\]
\end{definition}
\begin{proposition} \label{prop:Stab_h}
If $h=(h_n)$ is a monotonic sequence in $\N$ and $(A_n)\in\cS$, say  \eqref{eq:(An)alg} holds, then
\begin{eqnarray*}
\Stab_h(A_n) \ =\ \Big\{\ A,&& \lim_{i\to\infty}\sum_{j\in J_i}\prod_{k\in K_i} (P_-A_f^{(i,j,k)}P_-),\qquad \lim_{i\to\infty}\sum_{j\in J_i}\prod_{k\in K_i} (P_+A_g^{(i,j,k)}P_+)\\
&& \qquad:\quad A_f^{(i,j,k)}\in\Lim_h(A^{(i,j,k)}), \quad A_g^{(i,j,k)}\in\Lim_{-h}(A^{(i,j,k)})\ \Big\}\,.
\end{eqnarray*}
For $A_n=P_nAP_n$, this simplifies to $\{A,\ P_-A_fP_-,\ P_+A_gP_+\ :\ A_f\in\Lim_h(A),\ A_g\in\Lim_{-h}(A)\}$.
\end{proposition}
For the study of the stability of subsequences $(A_{h_n})$ in \cite{Li:FSMsubs}, the ``unwanted'' layers of $\oplus A_n$ have been removed, destroying its triangular pattern. Because this pattern and its dominant directions play an important role in our paper, we instead replace the unwanted layers, $A_m$ with $m\not\in\im h$, by $cP_m$ with $c\in\R$ chosen to make these layers irrelevant for the current purpose.

\noindent
\begin{minipage}{83mm}
So fix $c\in\R$ and put
\[
A_{n,c}'\ :=\ \left\{
\begin{array}{cl}
A_n,&n\in\im h,\\
cP_n,& n\not\in\im h.
\end{array}
\right.
\]
Then
\begin{equation} \label{eq:StabAn'}
\cB(A_{n,c}')\ =\ \cB_h(A_n)\cup\{cI\}
\end{equation}
with identity $I$ as in Remark \ref{rem:whichI} and both
\begin{eqnarray} \label{eq:limsupAn'}
\limsup\|A_{n,0}'\|&=&\limsup\|A_{h_n}\|\\
\label{eq:limsupAn'-1}
\limsup\|A_{n,c}'^{-1}\|&=&\limsup\|A_{h_n}^{-1}\|
\end{eqnarray}
\end{minipage}
\begin{minipage}{82mm}
\begin{center}
\begin{tikzpicture}[scale=0.24]
\foreach \k in {2,3,5,7,11}
  \fill[gray!50] (-\k,\k) rectangle (\k+1,\k+1);
\draw[gray!30,very thin,step=1cm] (-14,0) grid (15,15);
\foreach \k in {1,2,...,12} {
   \draw [black,line width=0.5mm] (-\k,\k) rectangle (\k+1,\k+1); 
}   
\foreach \j in {1,2,3} {
  \filldraw (0.5,{13.0+0.5*\j}) circle (1.5pt);
  \filldraw ({13.0+0.5*\j},{13.0+0.5*\j}) circle (1.5pt);  
  \filldraw ({-12-0.5*\j},{13.0+0.5*\j}) circle (1.5pt);  
}
\node [text=gray] at (-11.4,1.5) { \Large $\bf \Z^2$};
\end{tikzpicture}\\[2mm]
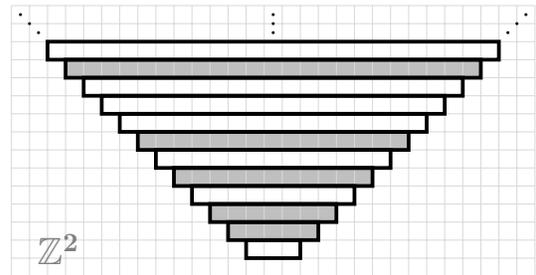
{\footnotesize Pattern of $\oplus A_{n,c}'$ with $h=(2,3,5,7,11,\dots)$.\\
Grey layers are original layers $A_n$, unchanged,\\[-1.2mm]
and white layers are $cP_n$ with a fixed $c\in\R$.
}
\end{center}
\end{minipage}
~\\
holding for large enough $c$, say $c>\sup\|A_n\|$, since then
\begin{equation} \label{eq:clarge}
c^{-1}\ <\ (\sup\|A_n\|)^{-1}\ =\ \inf\|A_n\|^{-1}\ \le\ \inf\|A_n^{-1}\| \le\ \inf\|A_{h_n}^{-1}\|\ \le\ \limsup\|A_{h_n}^{-1}\|.
\end{equation}
\begin{proof}[\bf Proof of \eqref{eq:limsupAhn}, \eqref{eq:limsupAhn-1}, \eqref{eq:limsupkappaAhn} and \eqref{eq:limsupSpepsAhn}]
It is enough to show \eqref{eq:limsupAhn} and \eqref{eq:limsupAhn-1}. The rest follows like Propositions \ref{prop:limsupSpeps} and \ref{prop:limsupkappa} follow from \ref{prop:limsupAn} and \ref{prop:limsupAn-1}. For $c>\sup\|A_n\|$,
\begin{eqnarray*}
\limsup\|A_{h_n}\| &\stackrel{\eqref{eq:limsupAn'}}=& \limsup\|A_{n,0}'\|\ \stackrel{\ci7}=\ \|\Lim(\oplus A_{n,0}')\|_\infty\ \stackrel{\ci8}=\ \|\cB_h(A_n)\|_\infty\\
\text{and}\quad
\limsup\|A_{h_n}^{-1}\| &\stackrel{\eqref{eq:limsupAn'-1}}=& \limsup\|A_{n,c}'^{-1}\|\ \stackrel{\ci7}=\ \|\big(\Lim(\oplus A_{n,c}')\big)^{-1}\|_\infty\ \stackrel{\ci9}=\ \|\big(\cB_h(A_n)\big)^{-1}\|_\infty\,.
\end{eqnarray*}
\begin{itemize}
\item[\ci7] This is by steps \ci1, \ci2, \ci5, \ci3 and \ci6 in the proofs of Propositions \ref{prop:limsupAn} and \ref{prop:limsupAn-1}. Note that these steps hold for all sequences in $\BDS$.

\item[\ci8] 
Let $\oplus B_n\in\Lim(\oplus A_{n,0}')$. W.l.o.g, just consider $B_0$.
By \eqref{eq:B0gen}, $B_0\ot S_{-i_k}A_{j_k,0}'S_{i_k}$ as $k\to\infty$.
\begin{itemize}
\item[a)]
If $j_k\in\im h$ eventually then $B_0\in\Lay(\Lim_{*\choose h}(\oplus A_n))$, by \eqref{eq:B_hAn}. If $(A_n)\in\cS$, further restrict to directions ${0\choose 1}$ and $ {\pm1\choose 1}$, i.e.~to $\Stab_h(A_n)$, without missing any maximizers.

\item[b)] 
If $j_k\not\in\im h$ eventually then $B_0=0$, clearly dispensable when maximizing $\|B_n\|$.

\item[c)] 
Also if $j_k$ keeps alternating between $\im h$ and its complement and the limit $B_0$ exists then $B_0=0$, as we see by passing to the subsequence $\not\in\im h$.
\end{itemize}

\item[\ci9] 
The argument is as in \ci8, only that now case b) and c) are of the form $B_0=cI$, which does not contribute to maximizing $\|B_n^{-1}\|$, by \eqref{eq:clarge}.\ \qedhere
\end{itemize} 
\end{proof}

\section{When is $\limsup$ a limit?} \label{sec:limsup=lim}
\begin{convention} \label{conv:f}
For $\s\in\{\|\cdot\|,\|(\cdot)^{-1}\|,\Speps\}$, we interpret $\s(A_n)$ as $\|A_n\|$ or $\|A_n^{-1}\|$ or $\Speps A_n$ and, correspondingly, $\s(\cB(A_n))$ as $\|\cB(A_n)\|_\infty$ or $\|\big(\cB(A_n)\big)^{-1}\|_\infty$ or $\cup_{B\in\cB(A_n)}\,\Speps B$.
Depending on the context, we interpret the signs $\prec, \preceq$ either as $<,\le$ or $\subsetneq,\subseteq$.
\end{convention}

Then, for monotonic sequences $h=(h_n)$ in $\N$, \eqref{eq:limsupAhn}, \eqref{eq:limsupAhn-1} and \eqref{eq:limsupSpepsAhn} can be summarized as
\begin{equation} \label{eq:limsupf}
\limsup \s(A_{h_n})\ =\ \s(\cB_h(A_n)),\qquad (A_n)\in\BDS.
\end{equation}
For the full sequence, $h=(h_n)=(n)$, we recover Propositions \ref{prop:limsupAn}, \ref{prop:limsupAn-1} and \ref{prop:limsupSpeps}.

So let $(A_n)\in\BDS$ and $\s\in\{\|\cdot\|,\|(\cdot)^{-1}\|,\Speps\}$, see Convention \ref{conv:f}.
Now, when is the $\limsup$ in \eqref{eq:limsupf} a limit? The answer is simple:
If a subsequence $h=(h_n)$ of the naturals exists, where $\limsup \s(A_{h_n})$ differs from $\limsup \s(A_n)$ then $\s(A_n)$ is not convergent; otherwise it is. So
let us search for monotonic sequences $h$ in $\N$ with
\begin{equation} \label{eq:<F}
\s(\cB_h(A_n))\ \stackrel{\eqref{eq:limsupf}}=\ \limsup \s(A_{h_n})\ \prec \ \limsup \s(A_n)\ \stackrel{\eqref{eq:limsupf}}=\ \s(\cB(A_n)).
\end{equation}
The search for promising subsequences $h$ of $\N$ is not as hopeless as it first seems. 
For a monotonic sequence $g$ in $\N$, write $h\subseteq g$ if $h$ is a subsequence of $g$ and note that, by \eqref{eq:B_hAn},
\begin{equation} \label{eq:mon}
h\subseteq g
\quad\text{implies}\quad
\cB_h(A_n)\ \subseteq\ \cB_g(A_n),
\quad\text{whence also}\quad
\s(\cB_h(A_n))\ \preceq\ \s(\cB_g(A_n)).
\end{equation}
This monotonicity, \eqref{eq:mon}, is guiding our search for sequences $h$ with \eqref{eq:<F}. 
\begin{definition} \label{def:mini}
We say that a monotonic sequence $g$ in $\N$ is {\sl minimizing} for the operator sequence $(A_n)\in\BDS$ and write $g\in m(A_n)$ if $\cB_h(A_n)=\cB_g(A_n)$ for all subsequences $h\subseteq g$. 
\end{definition}
\begin{example}
For pure finite sections, $A_n=P_nAP_n$ with $A\in\BDO$, and a monotonic sequence $f$ in $\N$, by $\Lim_f(A)\ne\varnothing$, first pass to a subsequence $g\subseteq f$ such that the limit operator $A_g$ exists and then to a subsequence $h\subseteq g\subseteq f$ such that also $A_{-h}$ exists. By Proposition \ref{prop:Stab_h}, 
\[
\cB_h(A_n)\ =\ \Stab_h(A_n)\ =\ \{A,\ P_-A_hP_-,\ P_+A_{-h}P_+\}
\]
cannot get smaller by passing to subsequences of $h$.
So $h$ is a minimizing sequence for $(A_n)$. 
\end{example}

This was a good practice, here is the general statement for $(A_n)\in\cS$:
\begin{corollary}\label{cor:minirich}
For $(A_n)\in\cS$, every monotonic sequence in $\N$ has a minimizing subsequence.
\end{corollary}
\begin{Proof}
For $(A_n)\in\cS$, say \eqref{eq:(An)alg} holds, enumerate the set of all involved $A^{(i,j,k)}$ by $A^{(1)}, A^{(2)}, \dots$. For a monotonic sequence $f$ in $\N$, first pass to $g^{(1)}\subseteq f$ such that $A^{(1)}_{g^{(1)}}$ and $A^{(1)}_{-g^{(1)}}$ exist, then to a subsequence $g^{(2)}\subseteq g^{(1)}\subseteq f$ such that $A^{(2)}_{g^{(2)}}$ and $A^{(2)}_{-g^{(2)}}$ exist, and so on. Then take $h=(h_n)\subseteq f$ with $h_n=g^{(n)}_n$ for $n\in\N$ and note that all limit operators $A^{(i,j,k)}_h$ and $A^{(i,j,k)}_{-h}$ exist, whence
\begin{equation} \label{eq:Stabhmin}
\cB_h(A_n) = \Stab_h(A_n) = \Big\{A,\ \lim_{i\to\infty}\sum_{j\in J_i}\prod_{k\in K_i} (P_-A_h^{(i,j,k)}P_-),\ \lim_{i\to\infty}\sum_{j\in J_i}\prod_{k\in K_i} (P_+A_{-h}^{(i,j,k)}P_+)\Big\}
\end{equation}
cannot get any smaller via subsequences of $h$, so that $h$ is a minimizing sequence for $(A_n)$.
\end{Proof}
\begin{proposition} \label{prop:minlim}
If $g$ is a minimizing sequence for $(A_n)\in\BDS$ then $\s(A_{g_n})$ is convergent.
\end{proposition}
\begin{Proof}
By Definition \ref{def:mini} and \eqref{eq:limsupf}, all subsequences of $\s(A_{g_n})$ have the same $\limsup$.
\end{Proof}
\begin{example} \label{ex:L-FSalg}
For $(A_n)\in\clos\alg\{P_nLP_n:L\text{ banded Laurent}\}$, see Example \ref{ex:Laurent}, already $\Stab(A_n)$ is of the minimal form \eqref{eq:Stabhmin} with limit operators $L_g=L$.
So $h=(h_n)=(n)$ is minimizing and all three spectral quantities $\s(A_n)$ converge, by Proposition \ref{prop:minlim}.
\end{example}
\begin{proposition} \label{prop:conv}
{\bf a) }
For $(A_n)\in\BDS$ and $\s\in\{\|\cdot\|,\|(\cdot)^{-1}\|,\Speps\}$, the quantity $\s(A_n)$ is convergent if and only if $\limsup \s(A_n)=\limsup \s(A_{g_n})$ for all minimizing sequences $g$ of $(A_n)$.

{\bf b) }
If $(A_n)\in\cS$, we have
\[
\begin{array}{rcl}
\limsup\|A_n\| &=& \max\limits_{g\in m(A_n)} \|\Stab_g(A_n)\|_\infty\,,\\
\limsup\|A_n^{-1}\| &=& \max\limits_{g\in m(A_n)} \|\big(\Stab_g(A_n)\big)^{-1}\|_\infty\,,\\
\limsup\kappa(A_n) &=& \Big(\max\limits_{g\in m(A_n)} \|\Stab_g(A_n)\|_\infty\Big)\cdot\Big(\max\limits_{g\in m(A_n)} \|\big(\Stab_g(A_n)\big)^{-1}\|_\infty\Big)\,,\\
\limsup\Speps A_n &=& \bigcup\limits_{g\in m(A_n)} \bigcup\limits_{B\in\Stab_g(A_n)}\Speps B.
\end{array}
\]
For the corresponding $\liminf$, replace $\max_g$ by $\min_g$ and $\cup_g$ by $\cap_g$.
\end{proposition}
\begin{Proof}
{\bf a) }
The implication $\boxed\Rightarrow$ is obvious and $\boxed\Leftarrow$ holds by monotonicity, \eqref{eq:mon}.

{\bf b) }
For $s\in\{\|\cdot\|,\|(\cdot)^{-1}\|,\kappa,\Speps\}$ and $(A_n)\in\cS$,
every partial limit $S$ of $s(A_n)$ is the limit of a subsequence $s(A_{h_n})$. By Corollary \ref{cor:minirich}, take a minimizing sequence $g\subseteq h$ for $(A_n)$ and note that
$S=\lim s(A_{h_n})=\lim s(A_{g_n})$. For $s=\kappa$, the latter is $\lim\|A_{g_n}\|\cdot\lim\|A_{g_n}^{-1}\|$ since all limits exist, by Proposition \ref{prop:minlim}. Now apply \eqref{eq:limsupf} and recall that $\limsup s(A_n)$ is, by definition, the largest such partial limit $S$, in the sense of the maximum or the set union, respectively.
\end{Proof}
\begin{corollary} \label{cor:div}
So the following are equivalent for $(A_n)\in\BDS$ and $\s\in\{\|\cdot\|,\|(\cdot)^{-1}\|,\Speps\}$:\\[-7mm]
\begin{enumerate}[label=(\roman*)\;] \itemsep-1mm
\item $\s(A_n)$ is divergent,
\item there exists a minimizing sequence $g$ of $(A_n)$ with $\limsup \s(A_{g_n})\prec \limsup \s(A_n)$,
\item there exists a minimizing sequence $g$ of $(A_n)$ with $\s(\cB_g(A_n))\prec \s(\cB(A_n))$\end{enumerate}
~\\[-7mm]
If $(A_n)\in\cS$, say \eqref{eq:(An)alg} holds, then all are equivalent to (iv): there is a monotonic sequence $g$ in $\N$, for which all limit operators $A^{(i,j,k)}_g$ and $A^{(i,j,k)}_{-g}$ exist and $\s(\Stab_g(A_n))\prec \s(\Stab(A_n))$. 
\end{corollary}
\begin{remark}
{\bf a) } Note that $\Stab_g(A_n)$ in Proposition \ref{prop:conv} b) and Corollary \ref{cor:div} $(iv)$ is of the simple form \eqref{eq:Stabhmin} with only three elements, say $\Stab_g(A_n)=\{A,B,C\}$, so that $\s(\Stab_g(A_n))$ is either $\max\{\|A\|,\|B\|,\|C\|\}$, $\max\{\|A^{-1}\|,\|B^{-1}\|,\|C^{-1}\|\}$ or $\Speps A\cup\Speps B\cup\Speps C$.

{\bf b) } If $\s(\Stab(A_n))=\s(A)$ then convergence of $\s(A_n)$ is clear without looking at subsequences, limit operators, etc. Recall, e.g., $\s=\|\cdot\|$ for pure finite sections, $A_n=P_nAP_n$.

{\bf c) } Convergence of $\Speps A_n$ implies that of $\|A_n^{-1}\|$. Indeed, putting $f_n(\lambda) := \mu(A_n-\lambda I_n)$ for $n\in\N,\,\lambda\in\C$, convergence of $\Speps A_n$ for all $\eps>0$ is equivalent to pointwise convergence of $f_n$, by \cite{LiSchmeck:Haus}, while convergence of $\|A_n^{-1}\|$ is just convergence of $f_n(0)$, by \eqref{eq:invmu}.

{\bf d) } The reverse implication of c) is not true, see the following example.
\end{remark}
\begin{example} \label{ex:shiftedflip}
Recall the symmetric block-flip $F$ from Example \ref{ex:symm_blockflip} with $\mu=0$
and put $A:=F+2I$, which is selfadjoint.
Then, with $F_n=P_nFP_n$ and $A_n=P_nAP_n$, $\ \Spec A_n=\Spec F_n+2$, which is $\{1,3\}$ if $n$ is even and $\{1,2,3\}$ if $n$ is odd.

Consequently, $\|A_n^{-1}\|=\mu(A_n)^{-1}=\dist(0,\Spec A_n)^{-1}=1$ for all $n$, which is constant (hence convergent), while
$\Speps A_n = \Spec A_n+\{z\in\C:|z|\le\eps\}$ differs between even and odd $n$ and does not Hausdorff-converge as $n\to\infty$.
\end{example}

\section{$\ell^p(\Z^d,Y)$: Banach space-valued $\ell^p$ over $\Z^d$ with $d\in\N,\ p\in[1,\infty]$} \label{sec:Banach}
\subsection{How to pass to $p\in\{1,\infty\}$}
A serious problem is that $P_n\not\to I$ pointwise if $p=\infty$, affecting the identification between operators and their matrix, the convergence $\|(P_n-I)T\|\to 0$ for compact operators $T$ on $\ell^\infty$, the somehow dual statement $\|T(P_n-I)\|\to 0$ on $\ell^1$, the whole interplay between matrix decay and compact operators with consequences reaching far into Fredholmness and limit operators for $\ell^1$ and $\ell^\infty$.  

The solution is to put the sequence $(P_n)$ in the first place and to change the notion of $\to$ such that $P_n\to I$ does hold, then change the set of compact operators $T$ such that $\|(P_n-I)T\|\to 0$ and $\|T(P_n-I)\|\to 0$ for all such $T$ and finally adapt the notion of Fredholmness to it. This major refurbishment is called $\cP$-theory \cite{RaRoSiBook}, and it received its ultimate polishing in \cite{Sei:Survey}. For $p\in(1,\infty)$, everything coincides with the classic theory, which is why we present that case above.

\subsection{Banach space-valued $\ell^p$ spaces}
Identifications like that of $L^p(\R)$ with the $L^p([0,1])$-valued $\ell^p(\Z)$ motivate to study spaces
\begin{equation} \label{eq:lpZY}
\ell^p(\Z,Y)\ :=\ \{(y_i)_{i\in\Z}\ :\ y_i\in Y,\ (\|y_i\|_Y)\in\ell^p(\Z)\}
\end{equation}
with a Banach space $Y$. Operators $A$ on such a space have matrices $(A_{ij})_{i,j\in\Z}$ with operator entries $A_{ij}$ on $Y$. Again, compactness of $A$ is only loosely related with decay properties of the matrix, as already one non-compact entry $A_{ij}$ can change everything. 
$\cP$-theory \cite{Sei:Survey} also comes to rescue here. Its redefinition of compact operators, convergence and Fredholmness is such that also $\ell^p(\Z,Y)$ smoothly integrates into the classic theory.

Exchanging $\ell^p(\Z)$ for $\ell^p(\Z, Y)$ in the previous sections requires two additional assumptions:\\[-7mm]
\begin{itemize}\itemsep-1mm
\item For many arguments we rely on the fact that $\Lim_h(A)\ne\varnothing$ for $A\in\BDO$ and a monotonic sequence $h$ in $\N$. The proof uses Bolzano-Weierstraß for the entries $A_{ij}$ and that requires scalar (or at least finite-dimensional) matrix entries. Indeed, $\Lim_h(A)=\varnothing$ can happen when $\dim Y=\infty$. To avoid this, restrict consideration to so-called {\sl rich} band-dominated operators $A$, literally imposing that $\Lim_h(A)\ne\varnothing$ for all $h$. Equivalently, $A\in\BDO$ is rich if and only if the set of all its entries $A_{ij}$ is relatively compact in $L(Y)$. 

\item Another property that comes automatic with $\dim Y<\infty$, see \cite{Globevnik,Shargorodsky08}, but has to be assumed additionally when $\dim Y=\infty$ is the continuity of the map $\eps\mapsto\Speps A$ for bounded operators $A$ on $X=\ell^p(\Z,Y)$. We rely on this property in Section \ref{sec:Speps}. It is shown to hold \cite{Globevnik,Shargorodsky08} for bounded linear operators $A$ on a Banach space $X$ that is finite-dimensional or complex uniformly convex or has a complex uniformly convex dual space. If we assume that $Y$ is uniformly $G$-convex in the sense of \cite{BoyKad} (and hence also complex uniformly convex) then, by \cite{BoyKad}, our $X=\ell^p(\Z,Y)$ inherits that property and we are back at continuity of $\eps\mapsto\Speps A$.
\end{itemize}
The last result, \cite{BoyKad}, is unfortunately only known to us for $p\in[1,\infty)$. So for our results on the asymptotics of $\speps A_n$ and $\Speps A_n$, we can (so far) either pass to $p=\infty$ or to $\dim Y=\infty$. For the rest of the paper both can happen at the same time.

\subsection{$\ell^p$-sequences over $\Z^d$ with $d\in\N$}
All our instruments \ci1-\ci9 are available and the arguments extend to $\ell^p$-spaces over $\Z^d$. The stacked operator $\oplus A_n$ then acts on $\ell^p(\Z,\ell^p(\Z^d))\cong \ell^p(\Z^{d+1})$.

As $P_n$ one could, for example, use multiplication by the characteristic function of the cube $(-n..n)^d$.  For $d=2$, the support pattern of $\oplus A_n$ is then an infinite upside down pyramid with square shaped layers. Besides $A$, the stability indicators $(c)$ and $(d)$ are limits along sequences $h={g_k\choose |g_k|_\infty}$ with $g_k\to\infty$ in $\Z^d$; that is, sequences $h$ on the surface of the pyramid in $\Z^{d+1}$. The counterparts of cases $(b),(e),(f)$ in Propositions \ref{prop:a-f1} and \ref{prop:a-f2} are again redundant.

We can still study the concept of minimal sequences $g$ but the corresponding set $\Stab_g(A_n)$ generally has uncountably many (instead of three) elements. This increase is not due to the method being inappropriate; instead, this growth is necessary to capture stability in higher dimensions. Already for $d=2$, a convolution operator $A$ and truncations $P_n$ to $(n\Omega)\cap\Z^2$ with a convex set $\Omega\subset\R^2$ with $0\in{\rm int}(\Omega)$, stability of the sequence $(P_nAP_n)$ is equivalent to invertibility of many compressions $P_UAP_U$ of $A$. For a polygon $\Omega$ with corners $v_1,\dots, v_k\in\Z^2$, the set $U$ needs to run through the $k$ limits, as $n\to\infty$, of $n(\Omega-v_j)$ for $j=1,\dots, k$, bringing us back to Kozak \cite{Koz}. For general $A\in\BDO$, one already has compressions $P_UA_hP_U$ of all kinds of limit operators $A_h$ with corresponding geometries $U$ that are no longer dominated by the $k$ operators with infinite cones $U$ corresponding to the corners of $\Omega$. For a disk $\Omega$, however, one cannot avoid looking at uncountably many compressions $P_UA_hP_U$ with half planes $U$.

\medskip

{\bf Acknowledgements.}
The authors thank Riko Ukena for helpful comments and discussions.

\vfill
\vfill
\noindent {\bf Authors' addresses:}\\
\\
Marko Lindner\hfill \href{mailto:lindner@tuhh.de}{\tt lindner@tuhh.de}\\
Dennis Schmeckpeper\hfill \href{mailto:dennis.schmeckpeper@tuhh.de}{\tt dennis.schmeckpeper@tuhh.de}\\
Institut Mathematik\\
TU Hamburg (TUHH)\\
D--21073 Hamburg\\
GERMANY
\end{document}